\documentclass[letterpaper,11pt]{article}
\usepackage{amsmath,amsfonts,amssymb,mathrsfs}
\usepackage{amssymb,mathrsfs,graphicx,enumerate}
\usepackage{amsthm,mathtools,amsbsy}
\usepackage[T1]{fontenc}

\usepackage[utf8]{inputenc}

\usepackage{subcaption}
\usepackage{graphicx,wasysym}

\usepackage[margin=1.15in]{geometry}
\numberwithin{equation}{section}
\mathtoolsset{showonlyrefs}

\usepackage{stmaryrd}
\usepackage{cases}
\usepackage{empheq}
\usepackage{paralist}
\usepackage{hyperref}
\usepackage{multimedia}
\usepackage{tikz}
\usepackage{enumitem}
\usepackage{cite}
\usepackage{multirow}
\usepackage{accents}

\newtheorem{thm}{Theorem}[section]
\newtheorem{prop}[thm]{Proposition}
\newtheorem{lem}[thm]{Lemma}
\newtheorem{cor}[thm]{Corollary}
\newtheorem{rem}[thm]{Remark}
\newtheorem{defn}[thm]{Definition}

\newcommand{\ap}[1]{\left\langle#1\right\rangle}
\def\grad{\nabla}
\DeclareMathOperator{\supp}{supp}
\def\der{\mathrm{d}}

\def\p{\partial}

\newcommand{\be}{\begin{equation}}
\newcommand{\ee}{\end{equation}}
\newcommand{\bes}{\begin{equation*}}
\newcommand{\ees}{\end{equation*}}

\newcommand{\norm}[1]{\left\Vert#1\right\Vert}
\def\st{\, \left|\right. \,}

\newcommand{\mt}[1]{\mathrm{#1}}
\def\R{\mathbb{R}}
\def\M{M}
\def\N{N}

\def\S{{\mathcal{D}_{\e}}}
\def\V{v}

\def\c{c}
\def\C{C}
\def\Lip{\Lambda} 

\def\PsiX{\Psi^t_X}
\def\PsiY{\Psi^t_Y}
\def\KU{K_{_U}}

\def\dRQ{\theta_{_{RQ}}}
\def\G{h}
\def\a{a}
\def\b{b}
\def\dim{n}
\def\tf{\phi}

\def\P{\mathcal{P}} 
\def\Cont{\mt{C}} 
\def\bd{\mathcal{W}}

\def\I{\mathcal{I}}
\def\U{\mathcal{U}}
\def\Q{\mathcal{Q}}
\def\calS{\mathcal{S}}
\def\calR{\mathcal{R}}
\newcommand{\bfv}{\mathbf{v}}
\def\:{\colon}
\def\d{\,\mathrm{d}}

\DeclareMathOperator{\Hess}{Hess}

\newcommand{\e}{\varepsilon}

\begin{document}
\title{An intrinsic aggregation model on the special orthogonal group $\mathrm{SO}(3)$: well-posedness and collective behaviours}

\author{Razvan C. Fetecau \thanks{Department of Mathematics, Simon Fraser University, Burnaby, BC V5A 1S6, Canada}
\and Seung-Yeal Ha \thanks{Department of Mathematical Sciences and Research Institute of Mathematics, Seoul National University, Seoul 08826 and Korea Institute for Advanced Study, Seoul 02455, Republic of Korea}
\and Hansol Park   \thanks{Department of Mathematical Sciences, Seoul National University, Seoul 08826, Republic of Korea}}

\date{\today}

\maketitle \centerline{\date}

\begin{abstract}
We investigate an aggregation model with intrinsic interactions on the special orthogonal group $SO(3)$. We consider a smooth interaction potential that depends on the squared intrinsic distance, and establish local and global existence of measure-valued solutions to the model via optimal mass transport techniques. We also study the long-time behaviours of such solutions, where we present sufficient conditions for the formation of asymptotic consensus. The analytical results are illustrated with numerical experiments that exhibit various asymptotic patterns.
\end{abstract}





\section{Introduction} We consider an aggregation model on a Riemannian manifold $\M$, given by:
\begin{equation}
\begin{cases}\label{eqn:model}
\displaystyle \partial_t\rho+\nabla_\M\cdot(\rho v)=0,\\
\displaystyle v=-\nabla_\M K*\rho,
\end{cases}
\end{equation}
where $K\: \M\times \M \to \R$ represents an interaction potential, and $\nabla_\M \cdot$ and $\nabla_\M $ denote the manifold divergence and gradient, respectively. The interaction potential $K$ is typically assumed to model short-range repulsive and long-range attractive interactions. Also, we use the symbol $\ast$ to denote a generalized convolution in the following sense: for a time-dependent measure $\rho_t$ on $\M$, the convolution $K \ast \rho_t$ is defined by
\begin{equation} \label{eqn:conv}
	K * \rho_t(x) = \int_\M K(x,y) \d \rho_t(y).
\end{equation}
In this paper we restrict $\rho_t$ to be a probability measure on $\M$, i.e., $\int_\M  \d \rho_t=1$ for all $t$.

Model \eqref{eqn:model} has been extensively studied in recent years, with the vast majority of these works concerning its Euclidean setup ($\M=\R^\dim$ with standard Euclidean metric). There exists a large body of literature on the mathematical analysis of solutions to model \eqref{eqn:model} in $\R^\dim$, which addresses the well-posedness of the initial-value problem  \cite{BertozziLaurent, Figalli_etal2011, BeLaRo2011}, the long time behaviour of solutions \cite{LeToBe2009, FeRa10, BertozziCarilloLaurent, FeHuKo11,FeHu13}, and the existence and characterization of minimizers for the associated interaction energy \cite{BaCaLaRa2013, Balague_etalARMA, ChFeTo2015}. At the same time a remarkable attention has been directed to the numerous applications of model \eqref{eqn:model}, e.g., biological swarms \cite{M&K}, material science and granular media \cite{CaMcVi2006}, self-assembly of nanoparticles \cite{HoPu2005}, robotics \cite{Gazi:Passino,JiEgerstedt2007} and opinion formation \cite{MotschTadmor2014}. Various qualitative features of swarm or self-organized behaviour have been captured with this class of models. Some inspiring collections of equilibria that can be obtained with this model can be found in \cite{KoSuUmBe2011, Brecht_etal2011} for instance, including aggregations on disks, annuli, rings, and soccer balls. 

Despite its remarkable potential for analysis and applications, literature on the aggregation model posed on surfaces or more general manifolds is far more limited. Before we review some of this literature, we want to distinguish two main classes of models that can be considered on manifolds. The first class consists of {\em extrinsic} models, which rely on a certain embedding of the manifold $\M$ in an ambient Euclidean space $\R^\dim$. In such case, the interaction potential $K(x,y)$ is taken to be of form $K(x,y) = K(|x- y|)$, where $|x-y|$ denotes the Euclidean distance in $\R^\dim$ between points $x$ and $y$ on $\M$. The other class is made of {\em intrinsic} models, which depend only on the intrinsic geometry of the manifold. For these models, the interaction potential is of the form $K(x,y) = K(d(x,y))$, where $d(x,y)$ is the geodesic distance on $\M$ between $x$ and $y$. In other words, models in the two classes consider extrinsic versus intrinsic interactions, respectively.

To elaborate on the type of interactions a little bit further, the interactions are encoded in the interaction potential $K$, more specifically in its gradient. For an extrinsic model, $\nabla_\M K(x,y) =K'(|x-y|) \nabla_{\M} |x-y|$, where $ \nabla_{\M} |x-y|$ can be found by projecting the Euclidean gradient $\nabla |x-y|$ on the tangent space of $\M$ at the point $x$ (note that these gradients are taken with respect to $x$ for $y$ fixed). On the other hand, for an intrinsic model, $\nabla_\M K(x,y)=K'(d(x, y)) \nabla_\M d(x, y)$, where $ \nabla_\M d(x, y) = -\frac{\log_x y}{d(x,y)}$ is expressed in terms of the Riemannian logarithm map on $\M$\cite{Petersen2006}. In particular, in intrinsic models, points $x$ and $y$ interact along the length minimizing geodesic curve between the two points, as opposed to interacting along the straight line connecting them in the ambient space, as for extrinsic models. In regions of high curvature, the distinction between the two types of interactions can be significant. Indeed, points that are close in Euclidean distance may be far apart in geodesic distance, and consequently, such points may repel each other in an extrinsic model (due to short-range repulsion), while they could attract themselves in an intrinsic model (by long-range attraction). The intrinsic approach appears more robust and more appropriate for applications. For example, consider applications in biology or engineering (robotics), where individuals/robots are restricted by environment or mobility constraints to remain on a certain manifold \cite{LiSpong2014,Li2015}. In such case, an efficient swarming must take into account inter-individual geodesic distances, and hence, intrinsic interactions.

Model \eqref{eqn:model} with extrinsic interactions has been studied in several works recently.  In \cite{WuSlepcev2015,CarrilloSlepcevWu2016}, the authors investigate the well-posedness of the aggregation model \eqref{eqn:model} on full-dimensional subsets of $\R^\dim$. Recently, several extrinsic Lohe-type models on the unit sphere, matrix manifolds and tensor spaces with the same rank and size were proposed in \cite{HaKiLeNo2019, HaKim2019, HaKoRy2017,HaKoRy2018,HaPark2020}. These models can be formulated as gradient flows for the square of the Frobenius norm of the average state. Note that the Frobenius metric is an extrinsic metric which can be obtained by embedding the given manifold into a larger Euclidean space. The emergent dynamics in such models has been studied extensively, and several sufficient frameworks for complete consensus and practical consensus were proposed. The proposed frameworks were formulated in terms of initial data and system parameters. 

The intrinsic model was investigated in \cite{FeZh2019}, with a focus on the emergent behaviour of its solutions on sphere and hyperbolic plane. It was shown there that solutions can approach asymptotically a diverse set of equilibria, such as constant density states, concentrations on geodesic circles, and aggregations on geodesic disks and annular regions. The well-posedness and asymptotic behaviour of  solutions to the intrinsic model on sphere was studied recently in \cite{FePaPa2020}. Applications that use the intrinsic properties of surfaces and manifolds have also been considered in the context of Cucker-Smale type models, another class of models for collective behaviours. Such models are second-order, as they are written in Newton's second law form. Cucker-Smale type models have been formulated and investigated recently on Riemannian manifolds, including  the unit sphere and hyperboloid, in \cite{AhnHaShim-1-2020, AhnHaShim-2-2020, HaKimSchloder2020}.
 
In this paper, we are exclusively concerned with the aggregation model set up on the $3$-dimensional special orthogonal group, that is, we take $\M = SO(3)$. The motivation for this choice lies in applications of the model in engineering, more specifically in robotics. Note that $SO(3)$ is the configuration space of a rigid body in $\R^3$ that undergoes rotations only (no translations). A group of robots engaged in self-organization by attractive/repulsive interactions can be modelled by the discrete/ODE analogues of \eqref{eqn:model} on the rotation group. An interesting engineering application for instance is to estimate the average pose of an object viewed by a network of cameras  \cite{TronVidalTerzis2008}. For the desired swarming behaviour, engineering works have focused mostly on two types of configurations: consensus (or synchronized) and anti-consensus (or balanced) states. The former type corresponds to a configuration where all agents occupy the same location (delta aggregation at a single point). Both extrinsic and intrinsic algorithms have been proposed and studied for achieving consensus on $SO(3)$ \cite{TronAfsariVidal2012, Markdahl2019}. The latter type of configurations corresponds to a group of robots well-distributed over a region/area, so that it achieves an optimal coverage needed for surveillance/tracking (the coverage problem) \cite{Sepulchre2011}. In this paper we will investigate in detail the first type of behaviour (consensus) in model \eqref{eqn:model} on $SO(3)$.
 
The goal of the present paper is two-fold. First, we establish the local and global well-posedness of solutions to model \eqref{eqn:model} on $SO(3)$. In this aim, we work with the geometric interpretation of model \eqref{eqn:model} as a continuity equation and consider weak, measure-valued solutions defined in the optimal mass transportation sense \cite{CanizoCarrilloRosado2011}. In geometric terms, model \eqref{eqn:model} represents the transport of the measure $\rho$ along the flow on $M$ generated by the tangent vector field $\V$, which depends on $\rho$ itself \cite{AGS2005}. This general framework enables us to include the discrete particle system as a particular case, and also study particle approximations and mean-field limits. The main result in this paper is Theorem \ref{thm:well-posedness}, which establishes the local well-posedness of solutions to model \eqref{eqn:model} on the rotation group. We also show in Theorem \ref{prop:inv-cont} that for purely attractive interaction potentials, solutions can be extended globally in time. A major aspect in this analysis lies in the regularity of the distance function, which is known to be non-smooth at the cut locus. For this reason we restrict the analysis to subsets of $SO(3)$ of diameter less than $\pi$, the injectivity radius of the rotation group. In particular, any pair of points in such subsets can be connected by a unique minimizing geodesic, ruling out ambiguities on how intrinsic interactions are defined. 

The second goal of the present work is to investigate the long-time behaviour of the solutions, in particular, the emergence of asymptotic consensus in model \eqref{eqn:model} on $SO(3)$. In literature, achieving such a state is also referred to as synchronization or rendezvous. As noted above, this represents an important problem in robotic control  \cite{Sepulchre2011,TronAfsariVidal2012, Markdahl2019}. Consensus states(or phase-locked states) have also been investigated for the Kuramoto oscillator and related models in \cite{ChiChoiHa2014, ChoiHaJungKim2012, ChopraSpong2009, HaKiPa2015, HaKimRyoo2016,HaKoRy2017,HaKim2019, Lohe2010, Lohe2009}. For surveys on related topics, we refer to \cite{DorflerBullo2014, HaKoRarkZhang2016} and references therein. For the applications of the model to opinion formation, we refer to \cite{MotschTadmor2014}. We will prove the formation of asymptotic consensus for the continuum model \eqref{eqn:model} on $SO(3)$ (Theorem \ref{thm:consensus-cont}), as well as refine the result for the specific case of the discrete model (Theorems \ref{thm:consensus-d1} and \ref{thm:consensus-d2}). We also present some numerical explorations of long-time behaviour and equilibrium solutions.
 
The rest of the paper is organized as follows. In Section \ref{sect:prelims}, we present some preliminaries, and set the notion of the solution and the assumptions on the interaction potential $K$. In Section \ref{sect:SO3}, we briefly discuss necessary background on the rotation group as a Riemannian manifold; in particular we present concepts such as geodesics and exponential/logarithm maps. In Section \ref{sect:well-posedness}, we establish the local well-posedness of solutions to model \eqref{eqn:model} on $SO(3)$, as well as their stability and mean-field approximation. In Section \ref{sect:global-cons}, we investigate the formation of asymptotic consensus for solutions to model \eqref{eqn:model} on the rotation group, for both the continuum and discrete formulations. In Section \ref{sect:numerics}, we present several numerical results. Finally, the Appendix is devoted to some concepts and results used in the main body of the paper to show the well-posedness and asymptotic behaviour.


\section{Preliminaries and general considerations}
\label{sect:prelims}
In this section, we present some background on flows on manifolds and Wasserstein distances, and then we introduce the notion of the measure-valued solution for model \eqref{eqn:model} and set up the assumptions on the interaction potential.

\paragraph{Flows on manifolds.} We briefly present  
some general facts for flows on manifolds. Although these facts hold for general manifolds, we restrict our discussion to $\M=SO(3)$. Denote by $\U$ a generic open subset of $SO(3)$, and consider a time-dependent vector field $X(R, t)$ on $\U\times [0, \a)$, for some $\a \in(0,\infty]$, i.e. $X_t(R):=X(R, t)\in T_R SO(3)$ for all  $(R,t)\in \U\times [0,\a)$. 

Given $\Sigma \subset \U$, a \emph{flow map} generated by $(X,\Sigma)$ is a function $\Psi_X\: \Sigma \times [0,\tau) \to \U$, for some $\tau\leq \a$, that satisfies:
\begin{equation} \label{eq:characteristics-general}
	\begin{cases} \dfrac{\der}{\der t} \Psi^t_X(R) = X_t(\Psi^t_X(R)),\\[5pt]
	\Psi^0_X(R) = R, \end{cases}
\end{equation}
for all $R\in\Sigma$ and $t\in[0,\tau)$, where we used the notation $\Psi^t_X$ for $\Psi_X(\cdot,t)$. A flow map is said to be \emph{maximal} if its time domain cannot be extended. Also, it is said to be \emph{global} if $\tau=\a=\infty$ and \emph{local} otherwise.

In the context of the present paper, the flow maps are generated by the velocity field $v[\rho]$ of the interaction equation (see \eqref{eqn:v-field} below), where $\Sigma$ is the support of the initial measure $\rho_0$. To simplify the terminology, unless there is potential for confusion, we will simply say that $v[\rho]$, instead of $(v[\rho],\supp(\rho_0))$, generates a flow map.

The local and global well-posedness of flow maps are covered by the standard theory of dynamical systems on manifolds; see \cite[Chapter~12]{Lee2013} or \cite[Chapter~4]{AMR1988} for instance. In the Appendix, we present briefly the results which we will need for our study. To establish the local well-posedness one needs to work in charts and make use of standard ODE theory in Euclidean spaces (see Theorem \ref{thm:Cauchy-Lip}). Note that here $\Sigma$ is assumed to be compact, as required for the maximal time of existence of the flow map to be strictly positive. We also present a global version of the Cauchy-Lipschitz theorem to be used in Section \ref{sect:global-cons}.


\paragraph{Notion of a solution.} In this paper we will interpret a solution $\rho_t$ of \eqref{eqn:model} as the push-forward of the initial density $\rho_0$ along the flow map generated by $\rho_t$ itself. To keep solutions as general as possible, we work with measure-valued densities; this framework will enable us to consider particle solutions and recover the discrete version of the model \eqref{eqn:model}. 

We denote by $\P(\U)$ the set of Borel probability measures on the metric space $(\U,d)$ and by $\Cont([0,T);\P(\U))$ the set of continuous curves from $[0,T)$ into $\P(\U)$ endowed with the narrow topology. Recall that a sequence $(\rho_n)_{n\geq 1} \subset \P(\U)$ \emph{converges narrowly} to $\rho \in \P(\U)$ if 
\bes
	\int_\U \tf(x) \d\rho_n(x) \to \int_\U \tf(x) \d\rho(x), \qquad \mbox{as $n\to\infty$, for all $\tf \in\Cont_\mt{b}(\U)$,}
\ees
where $\Cont_\mt{b}(\U)$ is the set of continuous and bounded functions on $\U$. 

We denote by $\Psi \# \rho$ the \emph{push-forward} in the mass transportation sense of $\rho$ through a map $\Psi\: \Sigma \to \U$ for some $\Sigma\subset \U$. Hence, $\Psi\#\rho$ is a probability measure on $\U$ such that for every measurable function $\zeta\: \U \to [-\infty,\infty]$ with $\zeta\circ \Psi$ integrable with respect to $\rho$, it holds that:
\bes
	\int_\U \zeta(x) \d (\Psi \# \rho)(x) = \int_\Sigma \zeta(\Psi(x)) \d\rho(x).
\ees

To recast model \eqref{eqn:model} in terms of transport along flow maps, we define for any curve $(\rho_t)_{t\in [0,T)} \subset \P(\U)$, the velocity vector field $\V[\rho]\: \U \times [0,T) \to T SO(3)$ associated to \eqref{eqn:model}, that is,
\begin{equation} \label{eqn:v-field}
	\V[\rho] (R,t) =  -\grad K *\rho_t (R) = \int_\U \grad K(R,Q) \d\rho_t(Q), 
\end{equation}
for all $(R,t) \in \U \times [0,T)$. For simplicity of notation, we have dropped the subindex $\M=SO(3)$ on $\nabla_M$. From here on, unless otherwise specified, $\nabla$ denotes the intrinsic (manifold) gradient on $SO(3)$. We also used $\rho_t$ in place of $\rho(t)$, as we shall often do in the sequel.  \newline

In this paper, we will adopt the following definition of weak (or measure-valued) solution of model \eqref{eqn:model} (see also \cite{CanizoCarrilloRosado2011}):
\begin{defn}[Weak solution]
\label{defn:solution}
	We say that $(\rho_t)_{t\in[0,T)} \subset \P(\U)$ is a \emph{weak solution} to \eqref{eqn:model} if $\V[\rho]$ generates a unique flow map $\Psi_{\V[\rho]}$ defined on $\supp(\rho_0)\times [0,T)$ and it holds that:
\be \label{eq:rho-push-forward}
	\rho_t = \Psi^t_{\V[\rho]} \# \rho_0, \qquad \mbox{for all $t\in[0,T)$}.
\ee
\end{defn}
It can be shown that a weak solution in the sense of Definition \ref{defn:solution} is also a distributional weak solution. In fact, following \cite[Lemma 8.1.6]{AGS2005} (see also \cite[Lemma 2.1]{FePaPa2020}), one has the following relation between a weak solution in the sense of \eqref{defn:solution} and a distributional weak solution.
\begin{lem}\label{lem:distrib-1}
Let $(\rho_t)_{t\in[0,T)} \subset \P(\U)$ be a weak solution of model \eqref{eqn:model} in the sense of Definition \ref{defn:solution}. In addition, we assume that $\V[\rho]$ satisfies
\bes
	\int_\calS \int_\Q {\| \V[\rho] (R,t) \|} \d\rho_t(R) \d t < \infty, \quad \text{ for all compact sets $\calS\subset (0,T)$ and $\Q \subset \U$}.
\ees
Then, $\rho_t$ is a distributional weak solution to \eqref{eqn:model}, i.e.,
\begin{equation}
   \int_0^T \int_\U \left( \p_t \tf(R,t) + {\langle \V[\rho] (R,t), \grad \tf(R,t) \rangle}_R \right)  \d\rho_t(R) \d t = 0,~~\mbox{for all $\tf \in \Cont_\mt{c}^\infty(\U \times (0,T))$}.
\end{equation}
\end{lem}


\paragraph{Wasserstein distance.} We compare solutions to \eqref{eqn:model} using the intrinsic $1$-Wasserstein distance on the rotation group. For $\rho,\sigma \in \P(\U)$, the intrinsic $1$-Wasserstein distance is given by:
\bes
	W_1(\rho,\sigma) = \inf_{\pi \in \Pi(\rho,\sigma)} \int_{\U\times\U} d(R,Q) \d\pi(R,Q),
\ees
where $\Pi(\rho,\sigma) \subset \P(\U\times\U)$ is the set of transport plans between $\rho$ and $\sigma$, i.e., the set of elements in $\P(\U\times\U)$ with first and second marginals $\rho$ and $\sigma$, respectively. 

We note here that in general, for $1$-Wasserstein distances one needs to use the set of probability measures on $\U$ with finite first moment, denoted by $\P_1(\U)$. By compactness of the rotation group however, $\P_1(\U) = \P(\U)$, and hence $(\P(\U),W_1)$  is a well-defined metric space. We further metrize the space $\Cont([0,T);\P(\U))$ with the distance defined by
\bes
	\bd_1(\rho,\sigma) = \sup_{t \in [0,T)} W_1(\rho_t,\sigma_t), \qquad \mbox{for all $\rho,\sigma \in \Cont([0,T);\P(\U))$}.
\ees

The following lemma holds on general Riemannian manifolds, but we present it here for the rotation group $SO(3)$. It lists various Lipschitz properties of flows of probability densities on $\U$ (a generic open subset of $SO(3)$) with respect to the $1$-Wasserstein distance.
\begin{lem}
\label{lem:preliminary}
	The following four statements hold.
	\begin{enumerate}[label=(\roman*)]
		\item\label{it:prel1} Let $\Sigma\subset \U$, $\rho \in \P(\U)$ with $\supp(\rho) \subset \Sigma$ and $\Psi_1,\Psi_2\:\Sigma \to \U$ be measurable functions. Then,
			\bes
				W_1({\Psi_1}\#\rho,{\Psi_2}\#\rho) \leq \sup_{R \in \supp(\rho)} d(\Psi_1(R),\Psi_2(R)).
			\ees
		\item\label{it:prel2} Let $\a\in(0,\infty]$ and  $X$ be a time-dependent vector field on $\U\times[0,\a)$, and $\rho \in \P(\U)$.  Suppose $(X,\supp(\rho))$ generates a flow map $\Psi_X$ defined on $\supp(\rho)\times[0,\tau)$ for some $\tau\leq \a$ and $X$ is bounded on $\U\times [0,\tau)$, i.e., there exists $C>0$ such that $\norm{X(R,t)}_{R\in \U}<C$ for all $R\in \U$ and $t\in[0,\tau)$. Then,
			\bes
				W_1({\Psi^t_X} \# \rho,{\Psi^s_X} \# \rho) \leq C|t-s|, \quad \quad \mbox{for all $t,s \in [0,\tau)$}.
			\ees
		\item\label{it:prel3} Let $\Sigma\subset\U$ and $\Psi\: \Sigma \to \U$ be Lipschitz continuous as a map from the metric space $(\Sigma,d)$ into the metric space $(\U,d)$; denote by $L_\Psi$ its Lipschitz constant. Then, for any $\rho,\sigma \in \P(\U)$,
			\bes
				W_1(\Psi\#\rho,\Psi\#\sigma) \leq L_\Psi W_1(\rho,\sigma).
			\ees
	\end{enumerate}
\end{lem}
\begin{proof}
Proofs of these statements are presented for general Riemannian manifolds in \cite[Lemma 2.3]{FePaPa2020}. We refer the reader to this reference, also noting that in our context, probability densities in $\P(\U)$ necessarily have compact support.
\end{proof}


\paragraph{Assumptions on the interaction potential.} For future reference, we list here the assumptions we make on the interaction potential $K$. First, we assume that the interactions are intrinsic, that is, $K:SO(3)\times SO(3) \to \R$ depends only on the intrinsic distance $d$ on $SO(3)$. In Section \ref{sect:SO3}, we provide necessary materials on the Riemannian manifold structure of the rotation group. Since the distance function is not differentiable on the diagonal $\{(R,Q) \in SO(3)\times SO(3)\mid R = Q\}$, we take $K$ to depend on the squared distance function instead. Specifically, we make the following assumption on the interaction potential:

\begin{enumerate}[label=\textbf{(H)}]
\item \label{hyp:K} $K\: SO(3) \times SO(3)\to \R$ has the form
	\begin{equation}
	\label{eqn:K-gen}
		K(R,Q) = g(d(R,Q)^2), \qquad \mbox{for all } R,Q\in SO(3),
	\end{equation}
	where $g\: [0,\infty) \to \R$ is differentiable, with locally Lipschitz continuous derivative. 
\end{enumerate}

In the sequel we use the notation $K_Q(R)$ for $K(R,Q)$ and $d_Q(R)$ for $d(R,Q)$. The notation is particularly useful when we take the gradient of $K$ or $d$ with respect to one of the variables. For example the gradient with respect to $R$ of $K(R,Q)$ will show as $\nabla K_Q(R)$. 

With the notation and convention above, the intrinsic gradient of the distance function can be expressed as:
\begin{equation}
\label{eqn:gradd}
	\nabla d_Q(R) = -\frac{\log_R Q}{d(R,Q)}, \qquad \mbox{for $R\neq Q$},
\end{equation}
where $\log_R Q$ denotes the Riemannian logarithm map (i.e., the inverse of the Riemannian exponential map) on $SO(3)$ \cite{Petersen2006}. By chain rule, one can then compute from \eqref{eqn:K-gen}:
\begin{equation}
\label{eqn:gradK-gen}
	\nabla K_Q(R) = -2 g'(d(R,Q)^2) \log_R Q.
\end{equation}

Equations \eqref{eqn:gradd} and \eqref{eqn:gradK-gen} hold only for matrices $R$ and $Q$ that are within the injectivity radius of $SO(3)$ to each other (or equivalently, for matrices that are not in the cut locus of each other). For this reason, our analysis will be restricted to subsets of $SO(3)$ of diameter less than the injectivity radius (the injectivity radius of the rotation group is $\pi$ -- see Section \ref{sect:SO3} for more details).


The interpretation of \eqref{eqn:model} as an aggregation model can be inferred from \eqref{eqn:v-field} and \eqref{eqn:gradK-gen}. Specifically, a rotation matrix $R$ interacts with another rotation matrix $Q$ through a force of magnitude proportional to $|g'(d(R,Q)^2|d(R,Q)$, and either moves towards $Q$ (provided $g'(d(R,Q)^2) >0$) or moves away from $Q$  (provided $g'(d(R,Q)^2) < 0$). The velocity field at location $R$, as computed with  \eqref{eqn:v-field}, takes into account all such contributions via the convolution.


\section{The rotation group as a Riemannian manifold}
\label{sect:SO3}

The rotation group $SO(3)$ consists of $3\times 3$ orthogonal matrices with determinant $1$, that is,
\[
SO(3) = \{ R \in \R^{3 \times 3}: R^T R = I \text{ and } \text{det }R = 1 \}.
\]
The tangent space to $SO(3)$ at a rotation $R \in SO(3)$ is given by
\[
T_R SO(3) = \{ R A: A \in \mathfrak{so}(3)\},
\]
where $\mathfrak{so}(3)$ is the Lie algebra of $SO(3)$ consisting of $3\times 3$ skew symmetric matrices. The Riemannian metric on the tangent space $T_R SO(3)$ is given by:
\begin{equation}
\label{eqn:metric}
g(R A_1,R A_2) = \frac{1}{2} \langle R A_1,R A_2\rangle_F =  \frac{1}{2} \langle A_1,A_2\rangle_F,
\end{equation}
for any $R A_1$, $R A_2 \in T_R SO(3)$, where $ \langle \cdot, \cdot \rangle_F$ denotes the Frobenius inner product. 
Consequently, in the norm induced by the metric, one has:
\begin{equation}
\label{eqn:length}
\| RA \|_{T_R SO(3)} = \frac{1}{\sqrt{2}} {\| RA \|}_F = \frac{1}{\sqrt{2}} {\| A \|}_F.
\end{equation}
Throughout the paper, for notational convenience, we will use the dot $\cdot$ to denote the inner product given by the Riemannian metric. Note that by \eqref{eqn:metric} it differs by a factor of $\frac{1}{2}$ from the Frobenius inner product $\langle \cdot, \cdot \rangle_F$. Also, we will use $|\cdot|$ for the norm of a tangent vector in the Riemannian metric; by \eqref{eqn:length} it differs by a factor of $\frac{1}{\sqrt{2}}$ from the Frobenius norm $\| \cdot\|_F$.



\paragraph{Angle-axis representation.}
Any rotation $R \in SO(3)$ can be identified via the exponential map with a pair $(\theta,\bfv) \in [0,\pi] \times S^2$, where $S^2$ denotes the unit sphere in $\mathbb{R}^3$. The pair $(\theta,\bfv)$ is referred to as the angle-axis representation of the rotation, where the unit vector $\bfv$ indicates the axis of rotation and $\theta$ represents the angle of rotation (by the right-hand rule) about the axis. The representation of $R$ in terms of $(\theta,\bfv)$ is given by Rodrigues's formula. To list it, we need the following common notation:
\begin{equation}
\label{eqn:vmat}
\widehat{\bfv} = 
\begin{bmatrix}
    0 & -v_3 & v_2 \\
    v_3 & 0 & -v_1 \\
    -v_2 & v_1 & 0
\end{bmatrix},
\end{equation}
for $\widehat{\bfv} \in \mathfrak{so}(3)$ corresponding to $\bfv = (v_1,v_2,v_3) \in \mathbb{R}^3$. Then, the angle-axis representation of a rotation $R$ is:
\begin{equation}
\label{eqn:tvtoR}
R = \exp(\theta \widehat \bfv ) = I + \sin \theta \, \widehat \bfv + (1-\cos \theta) \widehat \bfv^2,
\end{equation}
with $\widehat \bfv$ given by \eqref{eqn:vmat}. The inverse of the representation \eqref{eqn:tvtoR}, $\theta \widehat \bfv = \log R$, is given explicitly by:
\begin{equation}
\label{eqn:Rtotv}
\theta = \operatorname{acos}\left(\frac{\operatorname{tr}R-1}{2} \right), \qquad \widehat \bfv = \frac{1}{2 \sin \theta} (R - R^T).
\end{equation}
Here, $\exp$ and $\log$ represent the matrix exponential and logarithm, respectively.


\paragraph{Geodesic distance, exponential and logarithm maps.} Below, we list some standard facts on geodesics and the exponential map on the rotation group. Given two rotation matrices $R$, $Q \in SO(3)$, the shortest path between $R$ and $Q$ is along the geodesic curve $\calR :[0,1] \to SO(3)$ given by
\begin{equation}
\label{eqn:Rgeod}
\calR(t) = R \exp(t \log(R^T Q)).
\end{equation}
Note that $\calR'(t) = \calR(t) \log(R^T Q) \in T_{\calR(t)} SO(3)$. 

From \eqref{eqn:Rgeod}, one can easily see that the Riemannian distance on $SO(3)$ between $R$ and $Q$ is
\begin{equation}
\label{eqn:Rdist}
d(R,Q) = \dRQ,
\end{equation}
where $\dRQ \widehat{\bfv}_{_{RQ}}= \log(R^T Q)$. Throughout the paper we will frequently use the notation $\dRQ$ to denote the distance on $SO(3)$ between $R$ and $Q$. By \eqref{eqn:Rtotv} we also have:
\begin{equation}
\label{eqn:geod-dist}
d(R,Q)  =  \operatorname{acos}\left(\frac{\operatorname{tr}(R^T Q) -1}{2} \right),
\end{equation}
and in particular, $d(I,R) = \operatorname{acos}\left(\frac{\operatorname{tr}R -1}{2} \right)$. 

Using \eqref{eqn:Rgeod} one can also find explicitly the exponential map at $R$:
\begin{equation*}
\exp_{R}: T_{R} SO(3) \to SO(3), \qquad \exp_{R}(R A) = R \exp (A),
\end{equation*}
and its inverse:
\begin{equation}
\label{eqn:log-map}
\log_{R}: SO(3) \to T_{R} SO(3), \qquad \log_{R}(Q) = R \log(R^T Q).
\end{equation}
The formulas \eqref{eqn:tvtoR} and \eqref{eqn:Rtotv} can be expressed using the exponential map at the identity $I$. Indeed,
\[
R = \exp_I (\theta \widehat \bfv ), \qquad \theta \widehat \bfv = \log_I R.
\]

Note that the considerations above lead to:
\[
\theta = d(I,R) = {\| \theta \widehat \bfv \|}_{T_I SO(3)} = \frac{1}{\sqrt{2}} {\| \theta \widehat \bfv \|}_F,
\]
where for the last equality we used \eqref{eqn:length}. On the other hand, by \eqref{eqn:vmat} and the fact that $\bfv$ is a unit vector in $\mathbb{R}^3$ we have
\[
{\|\widehat \bfv \|}_F^2 = 2 |\bfv |^2 =2,
\]
making the equation above consistent. This justifies the coefficient $\frac{1}{2}$ in the Riemannian metric \eqref{eqn:metric}.


\paragraph{Injectivity and convexity radius.}The injectivity radius of the rotation group is $\pi$. To have a well-defined gradient of the distance function, we only consider in this paper subsets of $SO(3)$ where no two points are in the cut locus of each other. Examples of such sets are geodesic disks of radius $r<\pi/2$. For $r\in (0,\pi/2)$ we denote by 
\begin{equation}
\label{eqn:setS}
	D_r = \left\{ Q \in SO(3) \mid d(I,Q) < r \right\},
\end{equation}
the geodesic disk centred at the identity matrix of radius $r$. In general, $D_r(R)$ denotes the geodesic disk centred at $R$ of radius $r$. Note that the convexity radius of the rotation group is $\frac{\pi}{2}$, and hence any disk in $SO(3)$ of radius less than $\frac{\pi}{2}$ is geodesically convex. In particular, the maximum distance between any two points in a disk of radius less than $\frac{\pi}{2}$ is bounded by $\pi$, the injectivity radius.

To illustrate directly the singularity at injectivity radius of the exponential/logarithm map on $SO(3)$, we introduce the following notation:
\begin{equation}
\label{eqn:f}
f(\theta) = \frac{\theta}{\sin \theta}.
\end{equation}
By \eqref{eqn:log-map} and the angle-axis representation of $\log(R^TQ)$ (see \eqref{eqn:Rtotv}) one can then write:
\begin{equation} \label{eqn:gradd2}
\log_R Q =\frac{\dRQ}{2\sin\dRQ}R(R^TQ-Q^TR) =\frac{1}{2} f(\dRQ)(Q-RQ^TR).
\end{equation}
Note that $f(\dRQ) \to \infty$ as $\dRQ \to \pi$, the injectivity radius. In the sequel we fix $\e>0$ arbitrarily small and use the notation $\S$ for the disk $D_{\frac{\pi}{2}-\e}$, i.e.,
\[
\S =  \left\{ R \in SO(3) \mid d(I,R) < \frac{\pi}{2}-\e \right\}.
\]
This is the set on which we study and establish well-posedness of model \eqref{eqn:model}.  We chose $I$ as the centre of the disk with no loss of generality; the considerations in this paper would hold for a disk of radius $\frac{\pi}{2}-\e$ centred at a generic matrix $R$. 

Since $f$ and $f'$ are bounded on $[0,\pi-2 \e]$, we set
\begin{equation*}
	C_f := \sup_{\theta \in [0,\pi-2\e]} f(\theta), \qquad L_f := \sup_{\theta \in [0,\pi - 2\e]} f'(\theta).
\end{equation*}
Note that both $C_f$ and $L_f$ blow up as $\e \to 0$. Similarly, since the function $g'$ is assumed to be locally Lipschitz continuous, denote by
 $C_{g'}$ and $L_{g'}$ the $L^\infty$ norm and the Lipschitz constants of $g'$ on $[0,(\pi -2\e)^2]$, respectively. 
 
Note that for convenience of notations we chose not to indicate explicitly the dependence on $\e$ of these constants, a more pedantic notation would have been $C_f(\e),L_f(\e),C_{g'}(\e)$, and $L_{g'}(\e)$. We point out however that the dependence on $\e$ is essential and the results below do not hold in the limit $\e \to 0$.
 

\paragraph{Geodesic versus Frobenius distances.} All rotation matrices have constant Frobenius norm equal to $\sqrt{3}$. Hence, $SO(3)$ can be embedded as a subset of a sphere in the space $(\R^{3 \times 3},\langle \cdot, \cdot \rangle_F)$. For $R,Q \in SO(3)$, the distance $\|R-Q\|_F$ in the Frobenius norm relates to the geodesic distance $d(R,Q)$ as follows:
\begin{align}
{\|R-Q\|}_F^2&=\mathrm{tr}[(R-Q)^T(R-Q)]=\mathrm{tr}(2I-R^TQ-Q^TR)    \\
& =2\mathrm{tr}(I-R^TQ) =6-2\mathrm{tr}(R^TQ). \label{eqn:Fdist-ip}
\end{align}
From \eqref{eqn:geod-dist} one can then obtain:
\begin{equation}
\label{eqn:2dist}
d(R,Q)=\operatorname{acos} \left(1-\frac{1}{4}{\|R-Q\|}^2_F\right).
\end{equation}
Note that 
\begin{equation}
\label{eqn:ineq-dist}
{\| R - Q \|}_F \leq \sqrt{2} \, d(R,Q), \qquad \text{for all } R,Q \in SO(3).
\end{equation}
Indeed, one can use $\theta =  d(R,Q)$ and an elementary inequality $\sin \bigl( \frac{\theta}{2}\bigr) \leq  \frac{\theta}{2}$ for $\theta \in [0,\pi]$, to get 
\begin{equation}
\label{eqn:ineq-dist-expl}
{\| R - Q \|}_F^2 = 4 (1 - \cos \theta) = 8 \sin^2 \Bigl( \frac{\theta}{2}\Bigr) \leq 2 \, \theta^2.
\end{equation}


\section{Well-posedness of the intrinsic model on $SO(3)$}
\label{sect:well-posedness}
In this section, we establish the well-posedness of model \eqref{eqn:model} on $\S$, and also investigate the particle solutions and demonstrate the mean-field approximation.


\subsection{Vector fields on SO(3)}
\label{subsect:v-fields}
We first investigate some properties of flows on $SO(3)$ corresponding to a given vector field. We will make use of the fact that $SO(3)$ is embedded in $\R^{3 \times 3}$, which allows us to view tangent vectors to $SO(3)$ as vectors in $\R^{3 \times 3}$. In particular, one can then take the difference of tangent vectors at different points of $SO(3)$. In the following two lemmas below we will require that the vector fields satisfy a Lipschitz condition (see \eqref{eq:Lipschitz-X}) with respect to the Frobenius norm of the ambient space $\R^{3 \times 3}$. Subsequently in the paper (Lemma \ref{lem:dist-flow-maps-K}), we will show that the vector field associated to the interaction equation satisfies indeed this Lipschitz property.

\begin{lem}\label{lem:dist-flow-maps}
Let $X$ and $Y$ be two time-dependent vector fields on $\S$. Let $\Sigma\subset \S$ and suppose that $\Psi_X$ and $\Psi_Y$ are flow maps defined on $\Sigma\times [0,\tau)$, for some $\tau >0$, generated by $(X,\Sigma)$ and $(Y,\Sigma)$, respectively. Assume that $X$ is bounded on $\S\times[0,\tau)$ and Lipschitz continuous with respect to its first variable (uniformly with respect to $t \in [0,\tau)$) on $\S\times [0,\tau)$, i.e., there exists $L_X>0$ such that
\begin{equation}
\label{eq:Lipschitz-X}
{\| X(R, t)-X(Q, t) \|}_F \leq L_X \, d(R, Q),\quad \mbox{for all } (t, R, Q)\in[0, T]\times \S \times \S,
\end{equation}
where the difference $X(R,\cdot)-X(Q, \cdot)$ is considered in the ambient space $(\R^{3 \times 3},{\|\cdot\|}_F)$. Let $\PsiX$ and $\PsiY$ be the flow maps corresponding to $X$ and $Y$, respectively. Then, for all $R_0 \in \S$,
\[
d(\Psi_X^t(R_0), \Psi_Y^t(R_0))  \leq \frac{e^{C_\e t} -1}{C_\e} \|X-Y\|_{L^\infty(\S\times [0,\tau))},\quad \quad \mbox{for all $t\in[0,\tau)$},
\]
where 
\begin{equation}
\label{eqn:Ce}
C_\e =  \sqrt{6}\, \frac{\tan(\pi/2-\e)}{\pi/2-\e}  {\|X\|}_{L^\infty(\S \times [0,\tau))} + \frac{L_X}{\sqrt{2}}.
\end{equation}
\end{lem}
\begin{proof}
We fix $R_0 \in \S$ and estimate the distance $d(\Psi_X^t(R_0), \Psi_Y^t(R_0))$ as follows.
\begin{align*}
& \frac{d}{dt}d(\PsiX(R_0), \PsiY(R_0)) \\
&\qquad = \nabla d_{\PsiY(R_0)} (\PsiX(R_0)) \cdot X_t(\PsiX(R_0))+\nabla d_{\PsiX(R_0)} (\PsiY(R_0)) \cdot Y_t(\PsiY(R_0))\\
& \qquad = \frac{1}{2} \langle \nabla d_{\PsiY(R_0)}  (\PsiX(R_0)), X_t(\PsiX(R_0)) \rangle_F + \frac{1}{2} \langle \nabla d_{\PsiX(R_0)} (\PsiY(R_0)), Y_t(\PsiY(R_0)) \rangle_F.
\end{align*}
In what follows, we will add and subtract 
\[
\frac{1}{2} \langle \nabla d_{\PsiX(R_0)}(\PsiY(R_0)),X_t(\PsiX(R_0)) \rangle_F \quad \text{ and } \quad
\frac{1}{2} \langle \nabla d_{\PsiX(R_0)}(\PsiY(R_0)), X_t(\PsiY(R_0)) \rangle_F
\] 
to the right-hand side above. The reason for this is as follows. Note that the two vectors in the first inner product that we add and subtract, are tangent vectors to $SO(3)$ at {\em different} points. Hence, we used in this calculation the Frobenius inner product in the ambient space $(\R^{3 \times 3},{\|\cdot\|}_F)$. Therefore, one has 
\begin{align*}
\begin{aligned}
& \frac{d}{dt}d(\PsiX(R_0), \PsiY(R_0)) \\
&\quad = \underbrace{\frac{1}{2} \langle \nabla d_{\PsiY(R_0)}  (\PsiX(R_0)), X_t(\PsiX(R_0)) \rangle_F  + \frac{1}{2} \langle \nabla d_{\PsiX(R_0)}(\PsiY(R_0)),X_t(\PsiX(R_0)) \rangle_F}_{=:\I_1} \\
&\quad  \underbrace{-\frac{1}{2} \langle \nabla d_{\PsiX(R_0)}(\PsiY(R_0)),X_t(\PsiX(R_0)) \rangle_F + \frac{1}{2} \langle \nabla d_{\PsiX(R_0)}(\PsiY(R_0)), X_t(\PsiY(R_0)) \rangle_F}_{=:\I_2} \\
&\quad \underbrace{-\frac{1}{2} \langle \nabla d_{\PsiX(R_0)}(\PsiY(R_0)),X_t(\PsiY(R_0)) \rangle_F + \frac{1}{2} \langle \nabla d_{\PsiX(R_0)}(\PsiY(R_0)), Y_t(\PsiY(R_0)) \rangle_F}_{=:\I_3}.
\end{aligned}
\end{align*}
In the sequel, we estimate the terms ${\mathcal I}_i$ one by one. \newline

\noindent $\bullet$~(Estimate of ${\mathcal I}_3$): By direct calculation, one has 
\begin{equation} \label{eqn:I3}
\I_3  =   
\frac{1}{2} \langle \nabla d_{\PsiX(R_0)}(\PsiY(R_0)), Y_t(\PsiY(R_0)) - X_t(\PsiY(R_0)) \rangle_F   \\[2pt]
\leq \| X-Y\|_{L^\infty(\S \times [0,\tau))},
\end{equation}
where we used the Cauchy--Schwarz inequality and the fact that the gradient of the distance has norm equal to $1$; we also used \eqref{eqn:length} to relate the metric norm of $SO(3)$ with the Frobenius norm. \newline

\noindent $\bullet$~(Estimate of ${\mathcal I}_2$): Similarly, one has 
\begin{align}
\I_2 &=  \frac{1}{2} \langle \nabla d_{\PsiX(R_0)}(\PsiY(R_0)) , X_t(\PsiY(R_0)) - X_t(\PsiX(R_0)) \rangle_F   \\[5pt]
& \leq \frac{1}{\sqrt{2}} \| X_t(\PsiY(R_0)) - X_t(\PsiX(R_0)) \|_F   \\[5pt]
& \leq \frac{L_X}{\sqrt{2}} \, d(\PsiX(R_0),\PsiY(R_0)),
\label{eqn:I2}
\end{align}
where for the last inequality we used  the Lipschitz condition \eqref{eq:Lipschitz-X}. \newline

\noindent $\bullet$~(Estimate of ${\mathcal I}_1$):~Recall that 
\[
\I_1 =  \frac{1}{2} \langle \nabla d_{\PsiY(R_0)}  (\PsiX(R_0)) + \nabla d_{\PsiX(R_0)}(\PsiY(R_0)), X_t(\PsiX(R_0)) \rangle_F.
\]
We denote by $\theta=d(\Psi_X^t(R_0), \Psi_Y^t(R_0))$ the geodesic distance on $SO(3)$ between $\PsiX (R_0)$ and $\PsiY(R_0)$, and we set 
\[ R=\PsiX (R_0) \quad \mbox{and} \quad  Q=\PsiY (R_0). \]
Then we use \eqref{eqn:gradd2} to write:
\begin{equation} \label{eqn:I+A}
\I_1 =  - \frac{1}{2\theta}  \langle  \log_{R} Q+ \log_{Q} R, X_t(R) \rangle_F = - \frac{1}{4\sin\theta} \langle Q-R Q^TR+R-QR^TQ , X_t(R) \rangle_F.
\end{equation}
We use the commutativity of trace to find
\begin{align}
 \label{New-1} 
&{ \|  Q-R Q^TR+R-QR^TQ  \|}_F^2  \nonumber \\
& \hspace{0.5cm} = {\|R(R^TQ-Q^TR)+Q(Q^TR-R^TQ)\|}_F^2 ={\|(R-Q)(R^TQ-Q^TR)\|}_F^2 \nonumber \\
& \hspace{0.5cm}=\mathrm{tr}[(R-Q)(R^TQ-Q^TR)(R^TQ-Q^TR)^T(R-Q)^T] \nonumber \\
& \hspace{0.5cm} =\mathrm{tr}[(R^TQ-Q^TR)(R^TQ-Q^TR)^T(R-Q)^T(R-Q)] \nonumber \\
& \hspace{0.5cm}=\mathrm{tr}[(2I-R^TQR^TQ-Q^TRQ^TR)(R-Q)^T(R-Q)].
\end{align}
The terms inside the trace can be factored as follows.
\begin{align}
\label{New-2}
&2I-R^TQR^TQ-Q^TRQ^TR \nonumber \\
& \hspace{0.5cm} =4I-(I+R^TQR^TQ)-(I+Q^TRQ^TR) \nonumber \\
&  \hspace{0.5cm}= 4I-(I-R^TQ)(I-R^TQ)-(I-Q^TR)(I-Q^TR)-2R^TQ-2Q^TR \nonumber \\
&  \hspace{0.5cm} = 2(R-Q)^T(R-Q)+(R-Q)^TQR^T(R-Q)+(R-Q)^TRQ^T(R-Q) \nonumber \\
&  \hspace{0.5cm} = (R-Q)^T(2I+RQ^T+QR^T)(R-Q) \nonumber \\
&  \hspace{0.5cm} = (R-Q)^T(R+Q)(R+Q)^T(R-Q).
\end{align}
Now, we combine \eqref{New-1} and \eqref{New-2} to obtain
\begin{align*}
{\|Q-R Q^TR+R-QR^TQ\|}_F^2&=\mathrm{tr}[(R-Q)^T(R+Q)(R+Q)^T(R-Q)(R-Q)^T(R-Q)]\\
&\leq {\| R-Q \|}_F^4\cdot {\|R+Q \|}_F^2,
\end{align*}
where for the inequality we used Lemma \ref{L1}. 

On the other hand, we take a square root and use the fact that $R,Q$ have Frobenius norm $\sqrt{3}$ to get 
\[
{\|Q-R Q^TR+R-QR^TQ\|}_F \leq2\sqrt{3}{\|R-Q\|}_F^2 =16\sqrt{3} \sin^2 \frac{\theta}{2},
\]
where for the second line we used \eqref{eqn:ineq-dist-expl}. Then, we use the inequality above in \eqref{eqn:I+A} together with the Cauchy-Schwartz to get:
\begin{align}
\I_1  &\leq \frac{1}{4\sin\theta} {\|Q-R Q^TR+R-QR^TQ\|}_F \sup_{R \in \S} {\| X_t(R) \|}_F  \nonumber \\
 &\leq 2\sqrt{6}\tan\frac{\theta}{2} \| X \|_{L^\infty(\S \times [0,\tau))}. \label{eqn:I1}
\end{align}
Finally, we use 
\[ \tan\frac{\theta}{2} \leq \frac{\tan (\pi/2-\e)}{\pi/2-\e} \cdot \frac{\theta}{2}, \quad \mbox{for all $\theta \in [0,\pi-2 \e)$}, \]
and  combine all the estimates \eqref{eqn:I3}, \eqref{eqn:I2} and \eqref{eqn:I1} to get
\begin{align} 
&\frac{d}{dt}d(\Psi_X^t(R_0), \Psi_Y^t(R_0))  \leq \left( \sqrt{6}  \frac{\tan (\pi/2-\e)}{\pi/2-\e} \|X\|_{L^\infty(\S \times [0,\tau))} + \frac{L_X}{\sqrt{2}} \right) d(\Psi_X^t(R_0), \Psi_Y^t(R_0)) \nonumber \\
& \hspace{1cm} +||X-Y||_{L^\infty(\S \times [0,\tau))}. \label{eqn:est-total}
\end{align}
Then, Gronwall's lemma yields the desired estimate.
\end{proof}
In the next lemma, we establish a Lipschitz property for flows of vector fields satisfying \eqref{eq:Lipschitz-X}.
\begin{lem}\label{lem:Lipschitz-initial}
Let $X$ be a time-dependent vector field on $\S$. Let $\Sigma\subset \S$ and $\Psi_X$ be the flow map generated by $(X,\Sigma)$ on $[0,\tau)$, for some $\tau >0$. Suppose $X$ is bounded and satisfies \eqref{eq:Lipschitz-X}) on $\S\times [0,\tau)$. Then, 
\bes
	d(\PsiX(R),\PsiX(Q)) \leq e^{C_\e t} d(R,Q), \qquad \mbox{for all $R,Q \in \Sigma$ and $t\in [0,\tau)$},
\ees
with $C_\e$ given by \eqref{eqn:Ce}.
\end{lem}
\begin{proof}
Let $R,Q \in \Sigma$ be fixed. Now, we estimate the distance $d(\PsiX(R),\PsiX(Q))$ as follows.
\begin{align}
\begin{aligned} \label{eqn:estpq0}
& \frac{\der}{\der t} d(\PsiX(R),\PsiX(Q)) \\
& \hspace{1cm} = \nabla d_{\PsiX(Q)}(\PsiX(R)) \cdot X_t(\PsiX(R)) + \nabla d_{\PsiX(R)}(\PsiX(Q)) \cdot X_t(\PsiX(Q)). 
\end{aligned}
\end{align}
By adding and subtracting $\nabla d_{\PsiX(R)}(\PsiX(Q)) \cdot X_t(\PsiX(R))$ to the right-hand side of \eqref{eqn:estpq0}, we can proceed to estimate similar to $\I_1$ in Lemma \ref{lem:dist-flow-maps} (in particular, see \eqref{eqn:I1}), and obtain
\begin{align}
	& \left( \nabla d_{\PsiX(Q)}(\PsiX(R)) + \nabla d_{\PsiX(R)}(\PsiX(Q)) \right) \cdot X_t(\PsiX(R)) \nonumber  \\
	& \hspace{2cm} \leq 2\sqrt{6} \| X \|_{L^\infty(\S \times [0,\tau))} \tan \left( \frac{d(\PsiX(R),\PsiX(Q))}{2} \right) \nonumber \\
	& \hspace{2cm} \leq \sqrt{6} \| X \|_{L^\infty(\S \times [0,\tau))}  \frac{\tan (\pi/2-\e)}{\pi/2-\e} \, d(\PsiX(R),\PsiX(Q)). \label{New-3}
\end{align}
For the remaining two terms, we use the Cauchy--Schwarz inequality and the Lipschitz condition on $X$ to get
\begin{equation} \label{New-4}
\nabla d_{\PsiX(R)}(\PsiX(Q)) \cdot (X_t(\PsiX(Q)) - X_t(\PsiX(R))) \leq \frac{L_X}{\sqrt{2}} \, d(\PsiX(R),\PsiX(Q)).
\end{equation}
Finally, in  \eqref{eqn:estpq0}, we combine estimates \eqref{New-3} and \eqref{New-4} to find
\[
\frac{\der}{\der t} d(\PsiX(R),\PsiX(Q)) \leq \left(\sqrt{6} \| X \|_{L^\infty(\S \times [0,\tau))}  \frac{\tan (\pi/2-\e)}{\pi/2-\e} +  \frac{L_X}{\sqrt{2}} \right) d(\PsiX(R),\PsiX(Q)).
\]
This yields the desired estimate.
\end{proof}


\subsection{Well-posedness of solutions}
\label{subsect:well-posedness}

We first check that the vector field \eqref{eqn:v-field} associated to equation \eqref{eqn:model} is bounded and satisfies the Lipschitz condition \eqref{eq:Lipschitz-X}.

\begin{lem}\label{lem:dist-flow-maps-K}
 	Let $K$ satisfy \ref{hyp:K} and let $\rho \in \Cont([0,T);\P(\S))$. Then, the following assertions hold:  \\[2pt] 
\indent i) The vector field $\V[\rho]$ given by \eqref{eqn:v-field} is bounded on $\S\times [0,T)$:
\[
	\|\V[\rho]\|_{L^\infty(\S\times[0,T))} \leq 2 \pi C_{g'}.  \\[2pt]
\]
\indent ii) $\V[\rho]$ satisfies the Lipschitz condition \eqref{eq:Lipschitz-X}: there exists $L>0$ such that
\[
		{\| \V[\rho](R,t) - \V[\rho](Q,t) \|}_F \leq L \, d(R,Q),~~\mbox{for all $(R,Q,t) \in \S \times \S\times [0,T)$},
\]
where the Lipschitz constant $L$ depends only on $C_f,L_f,C_{g'}$ and $L_{g'}$.
\end{lem}
\begin{proof}
(i)~The boundedness of $\V[\rho]$ follows immediately from \eqref{eqn:v-field} and the assumption on $K$. Indeed, for all $(R,t) \in \S\times [0,T)$,
\begin{equation}
\label{eqn:X-bound}
	|\V[\rho](R,t)| \leq  \int_{\S} |2g'(d(R,Q)^2) \log_R Q| \d\rho_t(Q) \leq 2 \pi C_{g'}, 
\end{equation}
where we also used \eqref{eqn:gradK-gen}, the bound on $g'$ and that $|\log_R Q| = d(R,Q)<\pi$ for all $R,Q \in \S$. \newline

\noindent (ii)~To show the Lipschitz condition, let $R,Q \in \S$. By \eqref{eqn:v-field}, one has
\begin{equation}
\label{eqn:X-diff}
\V[\rho](R,t) - \V[\rho](Q,t) = \int_{\S} (\nabla \KU(R) - \nabla \KU (Q)) \d \rho_t(U),
\end{equation}
where the difference of the tangent vectors at different points $R$ and $Q$ is taken  in the embedding space $\R^{3 \times 3}$. For $R,Q,U \in \S$, we use notation \eqref{eqn:f} and expression \eqref{eqn:gradd2} (also recall notation \eqref{eqn:Rdist}) to find
\begin{align*}
\nabla d^2_U(R) - \nabla d^2_U(Q) &= -2 \log_R U + 2 \log_Q U \\
&= - f(\theta_{_{RU}}) (U - RU^T R) + f(\theta_{_{QU}})(U-QU^T Q).
\end{align*}
Again, we add and subtract $ f(\theta_{_{QU}}) (U - RU^T R)$ to the above relation and compute the resulting relation as 
\begin{align}
&{\| \nabla d^2_U(R) - \nabla d^2_U(Q) \|}_F \nonumber \\
& \hspace{1cm} \leq | f(\theta_{_{QU}}) -  f(\theta_{_{RU}}) | {\| U - RU^T R\|}_F + | f(\theta_{_{QU}})|  {\| RU^T R -  QU^T Q \|}_F   \nonumber \\
& \hspace{1cm} \leq  2 \sqrt{3} L_f | \theta_{_{QU}} -  \theta_{_{RU}}| + 2 \sqrt{3} C_f {\| R - Q \|}_F,
\label{eqn:est1}
\end{align}
where we used the Lipschitz property and bound of $f$ together with
\[
{\| U - RU^T R\|}_F \leq {\| U \|}_F + {\| RU^T R\|}_F = 2 \sqrt{3},
\]
and
\begin{align*}
{ \| RU^T R -  QU^T Q \|}_F &= {\| RU^T (R-Q) + (R-Q) U^T Q \|}_F \\
 & \leq  {\| RU^T \|}_F {\| R-Q \|}_F + {\| R-Q \|}_F {\| U^T Q \|}_F \\
 & = 2 \sqrt{3} {\| R-Q \|}_F.
\end{align*}
Then, from \eqref{eqn:est1}, using the triangle inequality $| \theta_{_{QU}} -  \theta_{_{RU}}| \leq \dRQ$, and $ \| R-Q \|_F \leq \sqrt{2} \, \dRQ$ (see \eqref{eqn:ineq-dist}) one gets:
\begin{equation}
\label{eqn:ineq-d}
{\| \nabla d^2_U(R) - \nabla d^2_U(Q) \|}_F \leq 2 \sqrt{3} (L_f + C_f \sqrt{2}) \, d(R,Q).
\end{equation}

Consider now an interaction potential in the form \eqref{eqn:K-gen}. For any $U\in\S$, we get:
\begin{align*}
	&{ \|\nabla K_U(R) - \nabla K_U(Q) \|}_F = {\|g'(d(R,U)^2) \nabla d_U^2(R) - g'(d(Q,U)^2) \nabla d_U^2(Q) \|}_F \\
	&\qquad \qquad \leq |g'(\theta_{_{RU}}^2) - g'(\theta_{_{QU}}^2) | {\|\nabla d_U^2(R)\|}_F + |g'(\theta_{_{QU}}^2)| {\|\nabla d_U^2(R) - \nabla d_U^2(Q)\|}_F,
\end{align*}
by adding and subtracting $g'(\theta_{_{QU}}^2) \nabla d_U^2(R)$ on the first line and then using triangle inequality. Further, by using the bounds and Lipschitz constants of $g'$, the fact that $|\nabla d_U^2(R)| = 2 \theta_{_{RU}}$, and \eqref{eqn:ineq-d}, we obtain:
\begin{align}
	 \|\nabla K_U(R) - \nabla K_U(Q)\|_F & \leq 2 L_{g'} |\theta_{_{RU}}+\theta_{_{QU}}||\theta_{_{RU}}-\theta_{_{QU}}| \theta_{_{RU}} +  2 \sqrt{3} C_{g'}(L_f + C_f \sqrt{2}) d(R,Q)   \\
	& \leq (4 \pi^2 L_{g'}+ 2 \sqrt{3} C_{g'}(L_f + C_f \sqrt{2})) \, d(R,Q), \label{eqn:ineq1-K}
\end{align}
where for the last inequality we used $|\theta_{_{RU}} - \theta_{_{QU}}| \leq d(R,Q)$ by triangle inequality, and that $\theta_{_{RU}},\theta_{_{QU}} < \pi$. Finally, we set
\[
L := 4 \pi^2 L_{g'}+ 2 \sqrt{3} C_{g'}(L_f + C_f \sqrt{2}),
\] 
and use \eqref{eqn:X-diff} and \eqref{eqn:ineq1-K} for all $t\in [0,T)$ to get 
\begin{equation*}
\label{eq:Lip-flow}
	\|\V[\rho](R,t) - \V[\rho](Q,t) \|_F \leq L d(R,Q) \int_{\S} \d\rho_t(U) = L\, d(R,Q),
\end{equation*}
where we also used that $\rho_t$ is a probability measure on $\S$.
\end{proof}

Another step used to establish the well-posedness of solutions is the following lemma; see \cite[Lemma~3.15]{CanizoCarrilloRosado2011}, and also \cite[Theorem~4.1]{CaChHa2014}.

\begin{lem}\label{lem:grad-lip-2}
	Let $K$ satisfy \ref{hyp:K} and let $\rho,\sigma \in \Cont([0,T);\P(\S))$. Then,
	\begin{equation}
		\|\V[\rho]-\V[\sigma]\|_{L^\infty([0,T)\times \S)} \leq \Lip \, \bd_1(\rho,\sigma),
	 \label{eqn:X-Yest}
	\end{equation}
	where $\Lip$ is a constant that depends only on $C_f$, $L_f$, $C_{g'}$ and $L_{g'}$.
\end{lem}
\begin{proof}
Take $R,Q,U \in \S$. By \eqref{eqn:gradd2}, one has:
\begin{align*}
\nabla d^2_Q(R) - \nabla d^2_U(R) &= -2 \log_R Q + 2 \log_R U \\
&= - f(\dRQ) (Q - RQ^T R) + f(\theta_{_{RU}})(U-RU^T R).
\end{align*}
Add and subtract $f(\theta_{_{RU}}) (Q - RQ^T R)$ to the above, to estimate:
\begin{multline}
\label{eqn:est2}
{\|  \nabla d^2_Q(R) - \nabla d^2_U(R) \|}_F \\
\leq | f(\theta_{_{RU}}) -  f(\dRQ) | {\| Q - RQ^T R\|}_F + | f(\theta_{_{RU}})|  {\| U-RU^T R - Q + RQ^T R \|}_F.
\end{multline}
For the second term in the right-hand-side, we use triangle inequality to get
\[
{ \| U-RU^T R - Q + RQ^T R \|}_F \leq {\| U-Q \|_F + \| R (Q - U)^T R \|}_F =2  {\| U-Q \|}_F.
\]
Then, from \eqref{eqn:est2}, using the Lipschitz property and bound of $f$, triangle inequality and the fact that a rotation matrix has Frobenius norm $\sqrt{3}$ we find:
\begin{align}
{\|  \nabla d^2_Q(R) - \nabla d^2_U(R) \|}_F  & \leq 2 \sqrt{3} L_f | \theta_{_{RU}} -  \dRQ| + 2 C_f(\e)  {\| U-Q \|}_F \nonumber  \\
& \leq  (2 \sqrt{3} L_f + 2 \sqrt{2} C_f)\, \theta_{_{QU}},
\label{eqn:ineq2}
\end{align}
where for the second inequality we also used $| \theta_{_{RU}} -  \dRQ| \leq  \theta_{_{QU}} $ by triangle inequality, and \eqref{eqn:ineq-dist}. For an interaction potential in the form \eqref{eqn:K-gen}, we compute:
	\begin{align*}
	\begin{aligned}
		 {\|\nabla K_Q(R) - \nabla K_U(R)\|}_F &= {\|g'(d(R,Q)^2) \nabla d^2_Q(R) - g'(d(R,U)^2) \nabla d^2_U(R)\|}_F \nonumber \\
		&\leq |g'(\dRQ^2) - g'(\theta_{_{RU}}^2)| {\|\nabla d^2_Q(R)\|}_F + |g'(\theta_{_{RU}}^2)| {\|\nabla d^2_Q(R) - \nabla d^2_U(R)\|}_F,	
	\end{aligned}
	\end{align*}
where we added and subtracted $g'(\theta_{_{RU}}^2) \nabla d^2_Q(R)$ on the first line and used triangle inequality. Then, using \eqref{eqn:ineq2}, the bound and Lipschitz constant of $g'$, and $|\nabla d^2_Q(R)|= 2 \dRQ$, we find:
\begin{align}
& {\|\nabla K_Q(R) - \nabla K_U(R)\|}_F  \nonumber \\
& \hspace{0.5cm} \leq 2L_{g'} |\dRQ+\theta_{_{RU}}| |\dRQ-\theta_{_{RU}}| \dRQ+ 2 C_{g'}  (\sqrt{3} L_f + \sqrt{2} C_f)d(Q,U) \nonumber \\
&  \hspace{0.5cm} \leq (4\pi^2 L_{g'}  + 2 C_{g'}  (\sqrt{3} L_f + \sqrt{2} C_f))d(Q,U),
 \label{eqn:ineq2-K}	
\end{align}
where for the second inequality we used $|\dRQ - \theta_{_{RU}}| \leq d(Q,U)$ by triangle inequality, and that $\dRQ,\theta_{_{RU}} <\pi$.

Now, for $(R,t) \in \S \times [0,T)$ arbitrary fixed, take $\pi_t \in \Pi(\rho_t,\sigma_t)$ to be an optimal transport plan between $\rho_t$ and $\sigma_t$, and estimate:
	\begin{align*}
		|\V[\rho](R,t) - \V[\sigma](R,t)| &= \left| \int_{\S} \nabla K_Q(R)\d\rho_t(Q) - \int_{\S} \nabla K_U(R) \d\sigma_t(U) \right|\\
		&= \bigg| \iint_{\S \times \S} \nabla K_Q(R) \d\pi_t(Q,U) - \iint_{\S \times \S} \nabla K_U(R) \d\pi_t(Q,U) \biggr|\\
		&\leq \iint_{\S \times \S} |\nabla K_Q(R) - \nabla K_U(R)| \d\pi_t(Q,U).
	\end{align*}
Then, using \eqref{eqn:ineq2-K} and \eqref{eqn:length}, we find:
	\begin{equation} \label{est:LipX}
		|\V[\rho](R,t) - \V[\sigma](R,t)| \leq \Lip  \iint_{\S \times \S} d(Q,U) \d\pi_t(Q,U) 
= \Lip W_1(\rho_t,\sigma_t)  \leq \Lip \bd_1(\rho,\sigma), 
	\end{equation}
where
\[
\Lip:= \frac{1}{\sqrt{2}} \left( 4 \pi^2 L_{g'} + 2 C_{g'}  (\sqrt{3} L_f + \sqrt{2} C_f) \right).
\]
We take the supremum in $(R,t)\in \S \times [0,T)$ on the left-hand side of \eqref{est:LipX} to derive the desired result.
\end{proof}

\begin{rem}
It is important to note that the upper bound on $\|\V[\rho]\|_{L^\infty(\S\times[0,T))}$ and the Lipschitz constant of $\V[\rho]$ in Lemma \ref{lem:dist-flow-maps-K}, as well as the Lipschitz constant $\Lip$ in Lemma \ref{lem:grad-lip-2}, do not depend on the densities $\rho$ and $\sigma$. This is a key observation used in the proofs of the local and global well-posedness of solutions.
\end{rem}

We now present the local well-posedness of solutions to model \eqref{eqn:model} on $SO(3)$, which is the main result of this section. The structure of the proof is based on the fixed-point argument used  by Canizo {\em et al.} \cite{CanizoCarrilloRosado2011} to prove the analogous result in the Euclidean case.
\begin{thm}[Well-posedness on $SO(3)$]\label{thm:well-posedness}
	Suppose that $K$ satisfies \ref{hyp:K} and let $\rho_0 \in \P(\S)$. Then, there exist $T>0$ and a unique weak solution in $\Cont([0,T);\P(\S))$ to the aggregation model \eqref{eqn:model} starting from $\rho_0$.
	\end{thm}
\begin{proof}
Relevant for this proof are Theorem \ref{thm:Cauchy-Lip} and Lemma \ref{lem:interaction-complete} presented in Appendix. Fix a curve $\sigma(t)$ in $\P(\S)$. By Lemma \ref{lem:interaction-complete}, the interaction velocity field $\V[\sigma]$ is locally Lipschitz and hence it defines a local flow on $\S$. The maximal time of existence for this flow map does not depend on $\sigma$, as noted in Remark \ref{rem:indep-max-time}. Consequently, there exists a maximal time $\tau>0$ such that the map $\Gamma$, given by
	\begin{equation}
	\label{eqn:Gamma}
		\Gamma(\sigma)(t) = {\Psi_{\V[\sigma]}^t}\# \rho_0, \qquad \mbox{for all $\sigma \in \Cont([0,\tau);\P(\S))$ and $t \in [0,\tau)$},
	\end{equation}
	is well-defined, where $\Psi_{v[\sigma]}$ is the unique flow map generated by $(v[\sigma],\supp(\rho_0))$ and defined on $\supp(\rho_0)\times [0,\tau)$. The goal is to show that $\Gamma$ is a map from $\Cont([0,\tau);\P(\S))$ into itself and that it has a unique fixed point.
	
Fix $\sigma \in \Cont([0,\tau);\P(\S))$. By Theorem \ref{thm:Cauchy-Lip} we have $\Psi_{v[\sigma]}^t(x)\in \S$ for all $x\in\supp(\rho_0)$ and $t\in[0,\tau)$. Hence $\Gamma(\sigma)(t)$ is supported in $\S$ and moreover, by conservation of mass, $\Gamma(\sigma)(t)$ is a probability measure on $\S$ for all $t\in[0,\tau)$. Since the map $t \to \Gamma(\sigma)(t)$ is continuous (see Lemmas \ref{lem:dist-flow-maps-K} and \ref{lem:preliminary}\ref{it:prel2}), we conclude that $\Gamma$ maps $(\Cont([0,\tau);\P(\S)),\bd_1)$ into itself.
	
Next we show that $\Gamma$ is a contraction provided we restrict the final time $T\leq \tau$ as follows. Let $\rho, \sigma \in \Cont([0,\tau);\P(\S))$. Then, for all $t \in [0,\tau)$,
	\begin{align}
		W_1({\Psi_{\V[\rho]}^t}\#\rho_0,{\Psi_{\V[\sigma]}^t}\#\rho_0) &\leq \sup_{x \in\supp(\rho_0)} d(\Psi_{\V[\rho]}^t(x),\Psi_{\V[\sigma]}^t(x))   \\
		&\leq \alpha (t) \| \V[\rho] - \V[\sigma]\|_{L^\infty([0,\tau)\times \S)}  \\
		& \leq \alpha(t) \Lip \, \bd_1(\rho,\sigma),
		\label{eqn:estW1}
	\end{align}
	where for the first inequality we used Lemma \ref{lem:preliminary}\ref{it:prel1}, for the second inequality we used Lemmas \ref{lem:dist-flow-maps-K} and \ref{lem:dist-flow-maps} with
	\bes
	\alpha(t) =  \frac{e^{ \bigl(2\sqrt{6}\pi \, \frac{\tan(\pi/2-\e)}{\pi/2-\e} C_{g'} + \frac{L}{\sqrt{2}}\bigr) t} -1}{2\sqrt{6}\pi \, \frac{\tan(\pi/2-\e)}{\pi/2-\e} C_{g'} + \frac{L}{\sqrt{2}}},
	\ees
and for the last inequality we used Lemma \ref{lem:grad-lip-2}. Note that $\alpha(t)$ is increasing in $t$ and $\lim_{t \to 0} \alpha(t) = 0$. Hence, since $\Lip$ is independent of $t$, one can choose $T\leq \tau$ small enough such that 
\[
\alpha(t) \Lip < \alpha(T) \Lip < \overline{C}, \qquad \text{ for all } t \in [0,T),
\]
for some constant $\overline{C}<1$. 

By restricting $T$ according to inequality above, and by taking the supremum over $[0,T)$ in \eqref{eqn:estW1} we infer that:
	\bes
		\bd_1(\Gamma(\rho),\Gamma(\sigma)) \leq \overline{C}\,  \bd_1(\rho,\sigma).
	\ees
Since $\overline{C}<1$, this shows that the restriction of $\Gamma$ to $(\Cont([0,T);\P(\S)),\bd_1)$ is a contraction. Consequently, $\Gamma$ has a unique fixed point  $\rho \in \Cont([0,T);\P(\S))$, i.e.,
	\bes
		\rho_t = \Psi_{\V[\rho]}^t \# \rho_0, \qquad \mbox{for all $t \in [0,T)$},
	\ees
which is the desired weak solution of model \eqref{eqn:model}.
\end{proof}

\begin{rem}
\label{rem:global} The solution established in Theorem \ref{thm:well-posedness} can be extended in time as long as its support remains within the set $\S$. For purely attractive interaction potentials ($g'\geq 0$), we show in Proposition \ref{prop:inv-cont} below that $\S$ is an invariant set for the dynamics and hence, the well-posedness of solutions holds globally in time, i.e., $T=\infty$.
\end{rem}


\subsection{Particle solutions}
\label{subsect:particle}

The theory established in Section \ref{subsect:well-posedness} can be applied to particle solutions of model \eqref{eqn:model}. Specifically, we take a positive integer $\N$ and consider a collection of masses $m_i \in (0,1)$ and rotation matrices $R_{i}^0 \in \S$, $i =1,\dots,N$. The total mass of the particles is $1$, that is, ${\sum}_{i=1}^{\N} m_i = 1$. We introduce the empirical measure associated to this set of masses and points:
\begin{equation}
	\label{eq:atomic-initial}
	\rho^{\N}_0 =  \sum_{i=1}^{\N} m_i \delta_{R_{i}^0},
\end{equation}
and denote by $\rho^\N$ the solution to model \eqref{eqn:model} on the interval $[0,T)$ (as established by Theorem \ref{thm:well-posedness}) starting from $\rho^{\N}_0$.

It is a standard fact \cite{CanizoCarrilloRosado2011} that the solution $\rho^{\N}$ is the empirical measure associated to masses $m_i$ and trajectories $R_{i}(t)$, $i = 1,\dots,\N$, i.e.,
	\be\label{eq:atomic}
		\rho_t^{\N} = \sum_{i=1}^{\N} m_i \delta_{R_i(t)}, \qquad \mbox{for all $t\in [0,T)$},
	\ee
where the (unique) collection of trajectories $R_i\:[0,T) \to \S$ satisfies, for all $i\in \{1,\dots,\N\}$ and $t\in[0,T)$,
	\be\label{eq:characteristics-particles}
		\begin{cases}
			{\dot R}_i(t) = \V[\rho^{\N}](R_i(t),t),\\
			R_i(0) = R_{i}^0.
		\end{cases}
	\ee

An important result in applications is the approximation of a continuum measure by empirical measures, referred to as the mean-field approximation. We investigate this approximation below. First, we derive a stability result, analogous to \cite[Theorem 3.16]{CanizoCarrilloRosado2011}.

\begin{thm}[Stability]\label{thm:stability} Let $K$ be an interaction potential that satisfies \ref{hyp:K}. Consider two initial densities $\rho_0,\sigma_0 \in \P(\S)$, and let $\rho$ and $\sigma$ be weak solutions to \eqref{eqn:model} defined on $[0,T)$ starting from $\rho_0$ and $\sigma_0$, respectively. Then, there exist $T^*\in(0,T)$ and an increasing, bounded function $r(\e,\cdot)$ with $r(\e,0) = 1$ such that
\bes
	W_1(\rho_t,\sigma_t) \leq r(\e,t)W_1(\rho_0,\sigma_0),	\qquad \mbox{for all $t\in [0,T^*)$}.
\ees
\end{thm}
\begin{proof}
We set $\Sigma = \supp(\rho_0)\cup\supp(\sigma_0)$. Then, by Theorem \ref{thm:Cauchy-Lip} and Lemma \ref{lem:interaction-complete}, there exist unique maximal flow maps $\tilde \Psi_{v[\rho]}$ and $\tilde \Psi_{v[\sigma]}$ generated by $(v[\rho],\Sigma)$ and $(v[\sigma],\Sigma)$, respectively. Denote by $\tau_\rho>0$ and $\tau_\sigma>0$ the respective maximal times of existence, and set $T^* = \min(\tau_\rho,\tau_\sigma,T)$.  We use triangle inequality to bound, for any $t \in [0,T^*)$,
	\begin{align}
	\begin{aligned} \label{est:stab1}
		W_1(\rho_t,\sigma_t) &= W_1(\tilde \Psi_{\V[\rho]}^t \# \rho_0,\tilde \Psi_{\V[\sigma]}^t \# \sigma_0)  \\
		&\leq W_1(\tilde \Psi_{\V[\rho]}^t \# \rho_0,\tilde \Psi_{\V[\sigma]}^t \# \rho_0) + W_1(\tilde \Psi_{\V[\sigma]}^t \# \rho_0,\tilde \Psi_{\V[\sigma]}^t \# \sigma_0).
	\end{aligned}
	\end{align}
W apply Lemma \ref{lem:Lipschitz-initial} for $X=\V[\sigma]$, which is bounded and Lipschitz continuous with respect to its first variable by Lemma \ref{lem:dist-flow-maps-K}. Infer that the map $\tilde \Psi_{\V[\sigma]}^t$ is Lipschitz continuous on $\S$ with Lipschitz constant $e^{C_\e[\sigma]t}$, where we use notation $C_\e[\sigma]$ for the constant in \eqref{eqn:Ce} with $X=\V[\sigma]$ and $\tau = T^*$.  Then use Lemma \ref{lem:preliminary}\ref{it:prel1} and Lemma  \ref{lem:preliminary}\ref{it:prel3}) to bound above the first and second terms in the right-hand side of \eqref{est:stab1}, respectively:
\begin{align}
W_1(\tilde \Psi_{\V[\rho]}^t \# \rho_0, & \tilde \Psi_{\V[\sigma]}^t \# \rho_0)  + W_1(\tilde \Psi_{\V[\sigma]}^t \# \rho_0,\tilde \Psi_{\V[\sigma]}^t \# \sigma_0) \nonumber  \\
& \leq \sup_{R \in\supp(\rho_0)} d(\tilde \Psi_{\V[\rho]}^t(R),\tilde \Psi_{\V[\sigma]}^t(R)) + e^{C_\e[\sigma] t} W_1(\rho_0,\sigma_0).
 \label{est:stab2}
\end{align}
We further bound the first term in the right-hand-side of \eqref{est:stab2} as follows. Using the estimate \eqref{eqn:est-total} for the vector fields $\V[\rho]$ and $\V[\sigma]$, we integrate it with an integrating factor to find
\begin{align}
d(\tilde \Psi_{\V[\rho]}^t(R),\tilde \Psi_{\V[\sigma]}^t(R)) & \leq \int_0^t e^{C_\e[\sigma](t-s)} \| \V[\rho](\cdot,s) - \V[\sigma](\cdot,s) \|_{L^\infty(\S)} \d s  \nonumber \\
&\leq \Lip  \int_0^t e^{C_\e[\sigma](t-s)} W_1(\rho_s,\sigma_s) \d s,
\label{est:stab3}
\end{align}
for all $R \in \supp(\rho_0)$ fixed, where for the second inequality we used \eqref{est:LipX}. Then, we combine \eqref{est:stab1}, \eqref{est:stab2} and \eqref{est:stab3} and multiply the resulting relation by $e^{-C_\e[\sigma]t}$ to obtain
	\bes
		e^{-C_\e[\sigma]t} W_1(\rho_t,\sigma_t) \leq \Lip \int_0^t e^{- C_\e[\sigma] s} W_1(\rho_s,\sigma_s) \d s + W_1(\rho_0,\sigma_0).
	\ees
Gronwall's lemma yields
	\bes
		e^{-C_\e[\sigma] t} \, W_1(\rho_t,\sigma_t) \leq e^{\Lip t} \, W_1(\rho_0,\sigma_0).
	\ees
As a last step, we use the expression for $C_\e[\sigma]$ and the upper bound for $\|\V[\sigma]\|_{L^\infty}$ from Lemma \ref{lem:dist-flow-maps-K}, to reach the desired inequality by setting:
\begin{equation}
\label{eqn:r}
r(\e,t):= e^{\left(\Lip+2\sqrt{6} \pi\, \frac{\tan(\pi/2-\e)}{\pi/2-\e} C_{g'} + \frac{L}{\sqrt{2}}\right)t}.
\end{equation}
\end{proof}

The mean-field limit is given by the following theorem.
\begin{thm}[Mean-field limit]\label{thm:mfl}
	Let $K$ be an interaction potential that satisfies \ref{hyp:K}. Consider an initial density $\rho_0 \in \P(\S)$ and let ${(\rho_0^{\N})}_{\N \in\mathbb{N}} \subset\P(\S)$ be of the form \eqref{eq:atomic-initial}, such that
	\bes
		W_1(\rho_0^{\N},\rho_0) \to 0, \qquad \mbox{ as $\N\to\infty$}.
	\ees
	Suppose that there exists $T>0$ such that $\rho$ and $\rho^{\N}$ are the unique weak solutions to model \eqref{eqn:model} on $[0,T)$, starting from $\rho_0$ and $\rho_0^{\N}$, respectively, for all $\N\in\mathbb{N}$ (see also \eqref{eq:atomic}). Then,  there exists $T^*\in(0,T)$ such that
	\bes
		 \sup_{t\in[0,T^*)} W_1(\rho_t^{\N},\rho_t)  \to 0, \qquad \mbox{ as $\N\to\infty$}.
	\ees
\end{thm}
\begin{proof} We will use Theorem \ref{thm:stability} for $\rho_t^{\N}$ and $\rho_t$. Note that the function $r$ identified in \eqref{eqn:r} does not depend on the choice of the two densities. Therefore, by Theorem \ref{thm:stability}, we infer that there exists a strictly increasing, bounded function $r(\e,\cdot) \: [0,T^*) \to[0,\infty)$ such that
	\bes
		W_1(\rho_t^{\N},\rho_t) \leq r(\e,t) W_1(\rho_0^{\N},\rho_0), \qquad \mbox{for all $t\in[0,T^*)$ and $\N\in \mathbb{N}$}.
	\ees
The function $r(\e,\cdot)$ is bounded on $[0,T^*)$ and denote by $C_r(\e,T^*)>0$ such a bound. Then we get:
	\bes
		\sup_{t\in[0,T^*)} W_1(\rho_t^{\N},\rho_t) \leq C_r(\e,T^*) W_1(\rho_0^{\N},\rho_0) \to 0, \qquad \mbox{ as $\N\to\infty$},
	\ees
	which concludes the proof.
\end{proof}


\section{Global well-posedness and asymptotic behaviour}
\label{sect:global-cons}
In this section, we establish the global well-posedness of solutions and investigate the formation of asymptotic consensus, when the interaction potential is purely attractive, i.e., $g' \geq 0$.


\subsection{Invariant sets and global well-posedness}
\label{subsect:global}
We will show below that for attractive potentials, any closed disk in $\S$ is an invariant set for the dynamics and hence, the well-posedness from Theorem \ref{thm:well-posedness} can be extended globally in time.

\begin{prop}[Global well-posedness in continuum model]
\label{prop:inv-cont}
Let $K$ satisfy \ref{hyp:K} with $g'\geq0$, and suppose the initial datum $\rho_0 \in \P(\S)$ satisfies
\[ \supp(\rho_0)\subset \overline D_r \quad \mbox{for some $r<\pi/2-\e$}. \]
Then, there exists a unique global weak solution to \eqref{eqn:model} in $\Cont([0,\infty);\P(\S))$ that starts from $\rho_0$; moreover, $\supp(\rho_t) \subset \overline D_r$ for all $t\in[0,\infty)$.
\end{prop}
\begin{proof} We resort to the global version of the Cauchy-Lipschitz theorem included in the Appendix; see Theorem \ref{thm:global-Cauchy-Lip} and specifically, Lemma \ref{lem:interaction-complete-global} for how the theorem applies to the interaction velocity field. Abusing the notation, denote by $\P(\overline D_r)$ the set of Borel probability measures on $\S$ that are supported in $\overline D_r$. 

Consider the map
\bes
	\Gamma(\sigma)(t) = \Psi_{v[\sigma]}^t \# \rho_0, \qquad \mbox{for all $\sigma \in \Cont([0,\infty);\P(\overline D_r))$ and $t\in[0,\infty)$},
\ees
where $\Psi_{v[\sigma]}^t$ is the unique global flow map generated by $(\V[\sigma],\supp(\rho_0))$. By Lemma \ref{lem:interaction-complete-global} this map is indeed defined for all $t \geq 0$. Also, by Theorem \ref{thm:global-Cauchy-Lip}, $\Psi_{v[\sigma]}^t(R) \in {\overline D_r}$ for all $R \in \supp \rho_0$ and $t \geq 0$, which implies that $\Gamma(\sigma)(t)$ is compactly supported in ${\overline D_r}$ for all $t \geq 0$.

Following the argument in the proof of Theorem \ref{thm:well-posedness}, we get that $\Gamma$ is a map from $(\Cont([0,\infty);\P(\overline D_r)),\bd_1)$ into itself. One also can infer the existence of a time $T>0$ such that the restriction of $\Gamma$ to $(\Cont([0,T);\P(\overline D_r)),\bd_1)$ is a contraction. Then, by the same fixed-point theorem procedure, there exists a unique $\rho \in \Cont([0,T);\P(\overline D_r))$ such that
	\bes
		\rho_t = \Psi_{\V[\rho]}^t \# \rho_0, \qquad \mbox{for all $t\in [0,T)$}.
	\ees	
We note that by the proof of Theorem \ref{thm:well-posedness}, the time $T$ is independent of $\rho_0$. Therefore, one can restart the procedure at time $T$ and then iteratively patch solutions through time to get the existence of a unique weak solution in $\Cont([0,\infty);\P(\overline D_r))$.
\end{proof}

Proposition \ref{prop:inv-cont} has an immediate analogue for the discrete model \eqref{eq:characteristics-particles}:

\begin{prop}[Global well-posedness in discrete model]
\label{prop:inv-disc}
Let $K$ satisfy \ref{hyp:K} with $g'\geq0$. Take a positive integer $\N$ and consider a collection of masses $m_i \in (0,1)$ with total mass $1$, and rotation matrices $R_{i}^0 \in \overline D_r$ for some $r<\pi/2-\e$, $i =1,\dots,\N$. Then, there exist unique  trajectories $R_i(t)$, $i =1,\dots,\N$ satisfying $R_i(t)\in \overline D_r$ and
\be\label{eqn:model-disc}
	\begin{cases} {\dot R}_i = -\displaystyle \sum_{j=1}^{\N} m_j \nabla K_{R_j}(R_i), \quad t > 0, \\
	R_i(0) = R_i^0, \end{cases}
\ee
for all $i\in \{1,\dots,\N \}$ and $t\in[0,\infty)$.
\end{prop}
\begin{proof}
The global well-posedness follows from Theorem \ref{thm:global-Cauchy-Lip} and the fact that 
\bes
	\log_Q I   \cdot \nabla K_{V}(Q) \leq 0, \qquad \mbox{for all $Q\in \S\setminus D_r$ and $V\in \overline D_r$};
\ees
see the proof of Lemma \ref{lem:interaction-complete-global}.
\end{proof}


\subsection{Asymptotic consensus in the continuum model}
\label{subsect:consensus-cont}
We define the following energy functional:
\begin{equation}
\label{eqn:energy-cont}
E[\rho]=\frac{1}{2}\int_{\S}\int_{\S} K(R, Q) \d\rho(R)\d\rho(Q), \qquad \text{ for all } \rho \in \P(\S).
\end{equation}
Model \eqref{eqn:model} is a gradient flow with respect to this energy\cite{FeZh2019}. For a fixed weak solution $\rho$ to \eqref{eqn:model}, we denote $E(t) := E[\rho(t)]$ and $\V(R,t) := \V[\rho](R,t)$, where $\V[\rho]$ is the interaction velocity field given by \eqref{eqn:v-field}. We present first some simple considerations regarding the asymptotic behaviour of $E(t)$ and its derivatives.

\begin{lem}\label{lem:E-lim}
Consider an interaction potential $K$ that satsifies \ref{hyp:K} with $g'\geq0$ and $g'$ continuously differentiable on $[0,4r^2]$. Suppose the initial datum $\rho_0\in\P(\S)$ satisfies 
\[ \supp(\rho_0)\subset \overline D_r \quad \mbox{for some $r<\pi/2-\e$}, \]
and let $\rho\in\Cont([0,\infty);\P(\S))$ be a global weak solution to \eqref{eqn:model} with the initial datum $\rho_0$. Then, one has
\[ \exists~E_\infty := \lim_{t \to \infty} E(t) \qquad \mbox{and} \qquad \sup_{0 \leq t < \infty}|{\ddot E}(t)| < \infty. \]
\end{lem} 
\begin{proof}
\noindent (i)~We denote by $\Psi_\V$ the global flow map generated by $\V[\rho]$ on $\supp(\rho_0)$. Then the derivative $E'(t)$ can be calculated by the push-forward formulation of $\rho$, the chain rule and the symmetry of $K$, as follows:
\begin{align}
\begin{aligned} \label{eqn:dEdt-cont}
\frac{\d}{\d t}E(t)&=\frac{1}{2}\frac{\d}{\d t}\int_{\S}\int_{\S} K(\Psi_v^t(R), \Psi_v^t(Q))\d\rho_0(R) \d\rho_0(Q)  \\
&=\int_{\S} \nabla K*\rho_t(\Psi_v^t(R))\cdot \V(\Psi_v^t(R), t)\ d\rho_0(R)  \\
&=-\int_{\S} {| \V(\Psi_v^t(R), t)|}^2 \d\rho_0(R)  \\
&=-\int_{\S} {|\V(R, t)|}^2 \d\rho_t(R)\leq 0.
\end{aligned}
\end{align}
Note that by Lemma \ref{lem:dist-flow-maps-K}, ${\dot E}(t)$ is bounded, and the map $t\mapsto E(t)$ is bounded below (as $\supp(\rho_t)\subset\overline D_r$ and $K$ is bounded on compact sets). Moreover by \eqref{eqn:dEdt-cont} $t\mapsto E(t)$ is nonincreasing. Hence we have the first assertion. \newline

\noindent (ii)~For the second assertion, we use \eqref{eqn:dEdt-cont} to calculate ${\ddot E}(t)$:
\begin{equation} \label{eq:E-second}
	{\ddot E}(t) = - \frac{\der}{\der t}  \int_{\S} | \V(\Psi_{\V}^t(R),t)|^2 \d\rho_0(R) = -2 \int_\S \frac{\der}{\der t} \V(\Psi_{\V}^t(R),t) \cdot  \V(\Psi_{\V}^t(R),t) \d\rho_0(R).
\end{equation}
Next, the time derivative in the integrand of \eqref{eq:E-second} can be computed using the definition of $\V$ along with the push-forward formulation:
\be
\label{eq:E-second2}
\frac{\der}{\der t} \V(\Psi_{\V}^t(R),t) = - \int_\S  \frac{\der}{\der t} \grad K_{\Psi_{\V}^t(Q)}(\Psi_{\V}^t(R)) \d\rho_0(Q).
\ee
To show that ${\ddot E}$ is bounded, it is enough to show that the expression in  \eqref{eq:E-second2} is bounded (note that $\V$ is bounded by Lemma \ref{lem:dist-flow-maps-K}). This can be shown by applying the product and chain rules to compute the integrand in \eqref{eq:E-second2}. The calculation leads to terms involving $g'(d(R,Q)^2)$, $g''(d(R,Q)^2)$, as well as derivatives involving the distance function, specifically $\grad d_Q^2(R)$, $\Hess d_Q^2(R)$ and $ \der\log_R(Q)$. The former is bounded on $\overline D_r\times \overline D_r$ by the assumption on $g$. The latter is also bounded, as the map $(R,Q)\mapsto d_Q^2(R)$ is smooth on the compact (and geodesically convex) set $\overline D_r\times \overline D_r$. We conclude from these considerations that ${\ddot E}$ is bounded on $[0,\infty)$.
\end{proof} 
 
We use the result above and Barbalat's lemma to prove the following proposition.
\begin{prop}
\label{prop:Barbalat}
	Let $K$ satisfy \ref{hyp:K}, with $g'\geq0$ and $g'$ continuously differentiable on $[0,4r^2]$, and suppose the initial datum $\rho_0\in\P(\S)$ satisfies 
\[ \supp(\rho_0)\subset \overline D_r, \quad  \mbox{with $r<\pi/2-\e$}. \]
Let $\rho\in\Cont([0,\infty),\P(\S))$ be the global weak solution to \eqref{eqn:model} with the initial datum $\rho_0$. Then, one has
\[
\lim_{t\rightarrow\infty}\int_\S |\V(R, t)|^2 \d\rho_t(R) = 0.
\]
\end{prop}
\begin{proof}
	By Lemma \ref{lem:E-lim} we know that $E(t)$ has a finite limit as $t \to \infty$ and ${\ddot E}(t)$ is bounded on $[0,\infty)$. By Barbalat's lemma \cite{Barbalat1959} we then conclude that ${\dot E}(t) \to 0$ as $t \to \infty$. The conclusion now follows from the expression of ${\dot E}(t)$ given by \eqref{eqn:dEdt-cont}.
\end{proof}
For future reference we also list the following immediate corollary.
\begin{cor}\label{L1.1} With the assumptions and notations of Proposition \ref{prop:Barbalat}, one has
\[
\lim_{t\to\infty}\int_{\S}{\| \V(R, t)\|}_{_F} \d\rho_t(R)=0.
\]
\end{cor}
\begin{proof}  We use the relation  ${\| \V(R, t)\|}_F = \sqrt{2}\, |\V(R, t)|$ and the Cauchy-Schwarz inequality to get
\begin{align}
\int_{\S}\| \V(R, t)\|_F \d\rho_t(R)  &\leq \left(\int_{\S}{\|\V(R, t)\|}^2_F \d\rho_t(R)\right)^{1/2}\cdot|\S|^{1/2} \nonumber \\
& = \sqrt{2}  \left(\int_{\S}{ |\V(R, t)|}^2 \d\rho_t(R)\right)^{1/2}\cdot|\S|^{1/2}.
\label{NNN-1}
\end{align}
Then, by Proposition \ref{prop:Barbalat}, the right-hand-side of \eqref{NNN-1} tends to zero as $t \to \infty$, and we obtain the desired result.
\end{proof}

We now focus the attention on the asymptotic behaviour of solutions to the continuum model. By \eqref{eqn:v-field}, \eqref{eqn:gradK-gen} and \eqref{eqn:gradd2}, we express the vector field $\V$ as
\begin{align}
\V(R,t)&=2 \int_{\S} g'(d(R,Q)^2) \log_R Q \d\rho_t(Q) \nonumber \\
&=\int_{\S} \frac{d(R,Q)}{\sin d(R,Q)} g'(d(R,Q)^2)(Q-RQ^TR) \d\rho_t(Q). 
\label{eqn:v-SO3}
\end{align}
We make the following notation:
\begin{equation}
\label{eqn:G}
\G(s):= 
\begin{cases}
\displaystyle \frac{s}{\sin s} \, g'(s^2), \quad &\text{ for } s>0, \\[5pt]
\displaystyle g^{\prime}(0), \quad &\text{ for}~s = 0.
\end{cases}
\end{equation}
Note that $\G$ has the same sign as $g'$, which is assumed to be non-negative. Using this notation we rewrite $\V$ as
\[
\V(R,t)=\int_{\S} \G(d(R,Q)) (Q-RQ^T R)\d\rho_t(Q), \qquad \text{ for all } (R,t) \in \S \times [0,\infty).
\]
We also set
\begin{equation}
\label{eqn:C}
\C(R, t):=\int_{\S} \G(d(R, Q)) Q \d\rho_t(Q),
\end{equation}
and express $\V$ from \eqref{eqn:v-SO3} as:
\begin{equation}
\label{eqn:vC}
\V(R,t)=\C(R,t)- R \C(R,t)^T R.
\end{equation}
In this section, we will make the following assumptions on the function $\G$:
\begin{equation}\label{eqn:G-hyp1}
	\text{ For any } \a>0 \text{ there exists } 0<\b<\a, \text { such that }\G(b) > 0,
\end{equation}   
and 
\begin{equation} \label{eqn:G-hyp2}
	\G \text{ is non-decreasing, i.e., } h(s_2) \geq h(s_1) \text{ for all } s_2 \geq s_1.
\end{equation}

\begin{rem}
\label{rem:gG} 
In terms of the interaction function $g$, conditions \eqref{eqn:G-hyp1} and \eqref{eqn:G-hyp2} are satisfied provided one can find an arbitrarily small $s>0$ such that $g'(s)>0$, and $g'$ is non-decreasing. These properties are satisfied by a wide range of interaction potentials, including power-law potentials, see the examples discussed at the end of this section.
\end{rem}

For simplicity, we will omit from the calculations below the dependence on $t$ of $\C$, and we will reinstate it back when necessary.

\begin{lem}
\label{lemma:TrC-prop1} 
Let $\rho_t \in \P(\S)$ be fixed and assume $\G$ satisfies \eqref{eqn:G-hyp2}. Then $C$ defined by \eqref{eqn:C} (with dependence on $t$ omitted) satisfies
\begin{equation}
\label{eqn:tr-ineq}
\mathrm{tr}\left(\C(R_1)R_2^T+\C(R_2)R_1^T \right)\geq \mathrm{tr}\left(\C(R_1)R_1^T+\C(R_2)R_2^T\right) \quad \text{ for all } R_1,R_2 \in \S.
\end{equation}
\end{lem}
\begin{proof}
By definition \eqref{eqn:C} of $\C$, we have
\begin{align*}
&\C(R_1)R_2^T+\C(R_2)R_1^T-\C(R_1)R_1^T-\C(R_2)R_2^T\\
&\qquad \qquad = \int_{\S}\left(\G(d(R_1, Q)) Q(R_2^T-R_1^T)-\G(d(R_2, Q)) Q(R_2^T-R_1^T)\right) \d\rho_t(Q)\\
&\qquad \qquad = \int_{\S}( \G(d(R_1, Q))-\G(d(R_2, Q)))(Q R_2^T- Q R_1^T) \d\rho_t(Q).
\end{align*}
From this calculation, we derive
\begin{align}
\begin{aligned}  \label{eqn:C-diff}
&\mathrm{tr}\left(\C(R_1)R_2^T+\C(R_2)R_1^T-\C(R_1)R_1^T-\C(R_2)R_2^T\right) \\
& \hspace{1cm} = \int_{\S}( \G(d(R_1, Q))-\G(d(R_2, Q)))\left(\mathrm{tr}(Q R_2^T)-\mathrm{tr}(Q R_1^T)\right)\d\rho_t(Q).
\end{aligned}
\end{align}
By \eqref{eqn:G-hyp2} and since $d(R,Q)$ is a decreasing function of $\mathrm{tr}(R^T Q)$ (see \eqref{eqn:geod-dist}), one has
\[
(\G(d(R_1, Q))-\G(d(R_2, Q)))(\mathrm{tr}(Q R_2^T)-\mathrm{tr}(Q R_1^T))\geq 0.
\]
The conclusion now follows from this observation and \eqref{eqn:C-diff}.
\end{proof}

\begin{lem} 
\label{lemma:TrC-prop2} 
Let $\rho_t \in \P(\S)$ be fixed and assume $\G$ satisfies \eqref{eqn:G-hyp2}. Then $C$ defined by \eqref{eqn:C} (with dependence on $t$ omitted) satisfies
\begin{align}
\begin{aligned}\label{A-2}
&\mathrm{tr}\big( \C(R_1)R_1^T(R_1-R_2)(R_1-R_2)^T+\C(R_2)R_2^T(R_1-R_2)(R_1-R_2)^T\big)\\
&\hspace{1cm} \leq \sqrt{3}\left( \|\C(R_1)-R_1\C(R_1)^TR_1\|_{_F}+\|\C(R_2)-R_2\C(R_2)^TR_2\|_{_F}\right).
\end{aligned}
\end{align}
\end{lem}
\begin{proof}
By simple manipulations, we have
\begin{align}
\begin{aligned} \label{NNew-1}
&\C(R_1)R_2^T +\C(R_2)R_1^T -\C(R_1)R_1^T-\C(R_2)R_2^T\\
& \hspace{1cm} = (\C(R_1)R_2^T-\C(R_1)R_1^TR_2R_1^T)+(\C(R_2)R_1^T-\C(R_2)R_2^TR_1R_2^T) \\
& \hspace{1.2cm} -\C(R_1)R_1^T(I-R_2R_1^T)-\C(R_2)R_2^T(I-R_1R_2^T),
\end{aligned}
\end{align}
and on the other hand, the same left-hand-side can be rewritten as:
\begin{align}
\begin{aligned} \label{NNew-2}
&\C(R_1)R_2^T +\C(R_2)R_1^T-\C(R_1)R_1^T-\C(R_2)R_2^T\\
& \hspace{1cm} = -\C(R_1)R_1^T(I-R_1R_2^T)-\C(R_2)R_2^T(I-R_2R_1^T).
\end{aligned}
\end{align}
Finally, we combine \eqref{NNew-1} - \eqref{NNew-2} and apply trace, and also use Lemma \ref{lemma:TrC-prop1} to find
\begin{align}
\begin{aligned} \label{eqn:ineq1}
0&\leq2\mathrm{tr}\big(\C(R_1)R_2^T+\C(R_2)R_1^T-\C(R_1)R_1^T-\C(R_2)R_2^T\big)  \\
&=\mathrm{tr}\big( (\C(R_1)R_2^T-\C(R_1)R_1^TR_2R_1^T)+(\C(R_2)R_1^T-\C(R_2)R_2^TR_1R_2^T)\big)  \\
&\hspace{0.2cm} -\mathrm{tr}\big(\C(R_1)R_1^T(2I-R_1R_2^T-R_2R_1^T)+\C(R_2)R_2^T(2I-R_1R_2^T-R_2R_1^T)\big).
\end{aligned}
\end{align}
By properties of the trace and by H\"{o}lder inequality, we have
\begin{align}
\begin{aligned} \label{eqn:ineq2}
&\mathrm{tr}\big(\C(R_1)R_2^T-\C(R_1)R_1^TR_2R_1^T\big) \\
& \hspace{1cm} =\mathrm{tr}\big( \C(R_1)R_2^T-R_1R_2^TR_1\C(R_1)^T\big) =\mathrm{tr}\big( (\C(R_1)-R_1\C(R_1)^TR_1)R_2^T\big)\\
& \hspace{1cm} \leq \|\C(R_1)-R_1\C(R_1)^TR_1\|_{_F}  \|R_2\|_{_F},
\end{aligned}
\end{align}
and similarly, one has 
\begin{align}
\begin{aligned} \label{eqn:ineq3}
&\mathrm{tr}\big(\C(R_2)R_1^T-\C(R_2)R_2^TR_1R_2^T\big) \\
& \hspace{1cm} =\mathrm{tr}\big( \C(R_2)R_1^T-R_2R_1^TR_2 \C(R_2)^T\big) =\mathrm{tr}\big( (\C(R_2)-R_2 \C(R_2)^TR_2)R_1^T\big) \\
& \hspace{1cm} \leq\|\C(R_2)-R_2\C(R_2)^TR_2\|_{_F} \|R_1\|_{_F}. 
\end{aligned}
\end{align}
Now, combining \eqref{eqn:ineq1}, \eqref{eqn:ineq2} and \eqref{eqn:ineq3} we find:
\begin{align*}
\begin{aligned}
& \mathrm{tr}\big( \C(R_1)R_1^T(2I-R_1R_2^T-R_2R_1^T)+\C(R_2)R_2^T(2I-R_1R_2^T-R_2R_1^T)\big)\\
& \hspace{1cm} \leq \|\C(R_1)-R_1\C(R_1)^TR_1\|_{_F} \|R_2\|_{_F}+\|\C(R_2)-R_2 \C(R_2)^TR_2\|_{_F} \|R_1\|_{_F}.
\end{aligned}
\end{align*}
By factoring out the left-hand-side above, one can then get \eqref{A-2}, also using that a rotation matrix has Frobenius norm $\sqrt{3}$.
\end{proof}

\begin{prop}
\label{prop:cons-cont}
Let $K$ satisfy \ref{hyp:K} and $r<\pi/4$. Suppose that $g'\geq0$ and $g'$ is continuously differentiable on $[0,4r^2]$. Also assume that $\G$ satisfies \eqref{eqn:G-hyp1} and \eqref{eqn:G-hyp2} (see also Remark \ref{rem:gG}). Let $\rho_0\in\P(\S)$ be such that $\supp(\rho_0)\subset \overline D_r$, and $\rho\in\Cont([0,\infty);\P(\S))$ be the global weak solution to \eqref{eqn:model} starting from $\rho_0$ from Proposition \ref{prop:inv-cont}. Then, for a fixed $\delta>0$ satisfying $\G(\delta)>0$, one has
\begin{equation}
\label{eqn:limit}
\lim_{t\to\infty} \iint_{d(R_1,R_2)>2\delta} \d\rho_t(R_1) \d\rho_t(R_2)=0.
\end{equation}
\end{prop}
\begin{proof}
Fix $\delta>0$ such that $\G(\delta)>0$; note that by assumption \eqref{eqn:G-hyp1}, $\delta$ can be arbitrarily small. Note that by Proposition \ref{prop:inv-cont}, $\supp(\rho_t) \subset \overline D_r$, and in particular the diameter of $\supp(\rho_t)$ is less than $\pi/2$.  From now on, we reinstate the dependence on $t$ of $C(\cdot,t)$. We use Lemma \ref{lem:l-bound} from Appendix to estimate 
\begin{align}
\begin{aligned} \label{A-3}
&\mathrm{tr}\left(C(R_1,t)R_1^T(R_1-R_2)(R_1-R_2)^T\right) \\
& \hspace{0.5cm} =\int_{\S}\G(R_1, Q)\mathrm{tr}\left(QR_1^T(R_1-R_2)(R_1-R_2)^T\right)\d\rho_t(Q) \\
&  \hspace{0.5cm} \geq\int_{\S}\G(R_1, Q) \cos(d(Q,R_1))  \|R_1-R_2\|_{_F}^2 \d\rho_t(Q) \\
& \hspace{0.5cm} \geq \|R_1-R_2\|_{_F}^2 \int_{\S \setminus B_\delta(R_1)}\G(R_1, Q) \cos(d(Q,R_1))\d\rho_t(Q), 
\end{aligned}
\end{align}
for $R_1 \in \supp(\rho_t)$. Here in the last inequality we used that $\G(R_1,Q)$ and $\cos(d(Q,R_1))$ are nonnegative. By assumption \eqref{eqn:G-hyp2}, 
\[ \G(R_1,Q) \geq \G(\delta) \quad \mbox{for all $Q \in \S \setminus  B_\delta(R_1)$}. \]
Using this fact together with $d(R,Q) \leq 2r$ for all $R,Q \in \supp(\rho_t)$, we infer from \eqref{A-3} that
\begin{align}
\begin{aligned}\label{A-4}
&\mathrm{tr}\left(C(R_1,t)R_1^T(R_1-R_2)(R_1-R_2)^T\right) \\
& \hspace{1cm} \geq \|R_1-R_2\|_{_F}^2 \int_{\S\backslash B_\delta(R_1)}\G(\delta) \cos\left(2r\right) \d\rho_t(Q) \\
&\hspace{1cm}  =\cos(2r) \G(\delta) \|R_1-R_2\|_{_F}^2 \int_{\S\backslash B_\delta(R_1)} \d\rho_t(Q)\\
&\hspace{1cm} =\cos(2r)\G(\delta) \|R_1-R_2\|_{_F}^2 (1-\rho_t(B_\delta(R_1))).
\end{aligned}
\end{align}
A similar estimate can be derived with $R_1$ and $R_2$ interchanged. Then, we combine \eqref{A-4} (and its analogue with $R_1 \leftrightarrow R_2$) and \eqref{A-2} to derive
\begin{align*}
& \cos(2r) \G(\delta) \|R_1-R_2\|_{_F}^2 (2-\rho_t(B_\delta(R_1))-\rho_t(B_\delta(R_2)))  \\
& \hspace{1cm} \leq \mathrm{tr}(C(R_1,t)R_1^T(R_1-R_2)(R_1-R_2)^T+C(R_2,t)R_2^T(R_1-R_2)(R_1-R_2)^T) \\
& \hspace{1cm} \leq \sqrt{3}\left( \|\C(R_1,t)-R_1\C(R_1,t)^TR_1\|_{_F}+\|\C(R_2,t)-R_2\C(R_2,t)^TR_2\|_{_F}\right). 
\end{align*}
We integrate the above relation with respect to $R_1$ and $R_2$  to find
\begin{align}
\begin{aligned} \label{A-5}
&\cos(2r)\G(\delta) \int_\S\int_\S \|R_1-R_2\|_{_F}^2 \, (2-\rho_t(B_\delta(R_1))-\rho_t(B_\delta(R_2))) \d\rho_t(R_1) \d\rho_t(R_2)  \\
&\hspace{1cm} \leq \sqrt{3}\int_\S \int_\S \Big( \|\C(R_1,t)-R_1\C(R_1,t)^TR_1\|_{_F} \\
& \hspace{3cm} +\|\C(R_2,t)-R_2\C(R_2,t)^TR_2\|_{_F}\Big)\d\rho_t(R_1)\d\rho_t(R_2)\\
&\hspace{1cm} =2\sqrt{3} \int_\S \|\C(R,t)-R\C(R,t)^TR\|_{_F} \d\rho_t(R)  \\
&\hspace{1cm} =2\sqrt{3}\int_\S \|v(R,t)\|_{_F}\d\rho_t(R), 
\end{aligned}
\end{align}
where for the last equal sign we used \eqref{eqn:vC}. \newline

For rotation matrices $R_1,R_2$ such that
\[ d(R_1, R_2)>2\delta, \quad  B_\delta(R_1)\cap B_\delta(R_2)=\phi, \]
one has
\[
2-\rho_t(B_\delta(R_1))-\rho_t(B_\delta(R_2))=2-\rho_t(B_\delta(R_1)\cup B_\delta(R_2))\geq1.
\]
Then the left-hand-side in \eqref{A-5} can be estimated below as:
\begin{align}
&\cos(2r) \G(\delta) \int_\S\int_\S \|R_1-R_2\|_{_F}^2 \, (2-\rho_t(B_\delta(R_1))-\rho_t(B_\delta(R_2)))\d\rho_t(R_1)\d\rho_t(R_2)   \\[2pt]
&\geq  \cos(2r) \G(\delta) \iint_{d(R_1,R_2)>2\delta}\|R_1-R_2\|_F^2(2-\rho_t(B_\delta(R_1))-\rho_t(B_\delta(R_2)))\d\rho_t(R_1)\d\rho_t(R_2)    \\[2pt]
&\geq \cos(2r) \G(\delta)\iint_{d(R_1,R_2)>2\delta}\|R_1-R_2\|_{_F}^2 \d\rho_t(R_1)\d\rho_t(R_2). \label{A-6}
\end{align}
We combine \eqref{A-5} and \eqref{A-6} to get
\begin{equation}
\label{A-7}
\iint_{d(R_1,R_2)>2\delta}\|R_1-R_2\|_{_F}^2 \d\rho_t(R_1)\d\rho_t(R_2) \leq \frac{2\sqrt{3}}{\cos(2r) \G(\delta)} \int_\S \|v(R,t)\|_{_F} \d\rho_t(R).
\end{equation}
By \eqref{eqn:2dist}, we write
\[
\|R_1-R_2\|_{_F}^2=4-4\cos(d(R_1,R_2)),
\]
and use this to estimate further
\begin{align*}
\begin{aligned}
&\iint_{d(R_1,R_2)>2\delta}\|R_1-R_2\|_{_F}^2 \d\rho_t(R_1)\d\rho_t(R_2)  \\
& \hspace{1cm} \geq \iint_{d(R_1,R_2) >2\delta} (4-4\cos(2\delta)) \d\rho_t(R_1)\d\rho_t(R_2) =8\sin^2(\delta) \iint_{d(R_1,R_2)>2\delta}\d\rho_t(R_1)\d\rho_t(R_2).
\end{aligned}
\end{align*}
Finally, we combine this and \eqref{A-7} to find
\begin{align*}
0\leq \iint_{d(R_1,R_2) >2\delta} \d\rho_t(R_1)\d\rho_t(R_2)\leq \frac{\sqrt{3}}{4\cos(2r)\sin^2(\delta)\G(\delta)}\int_\S \|v(R,t)\|_{_F}\d\rho_t(R).
\end{align*}
The conclusion now follows from Corollary \ref{L1.1}.
\end{proof}
Now, we are ready to prove the main result of this section.
\begin{thm}[Asymptotic consensus in the continuum model]\label{thm:consensus-cont}
	Let $K$ satisfy \ref{hyp:K} and $r<\pi/4-\e$. Suppose:  \\[2pt]
\indent i) $g'\geq0$ and $g'$ is continuously differentiable on $[0,4r^2]$,  \\[2pt]
\indent ii) $\G$ satisfies \eqref{eqn:G-hyp1} and \eqref{eqn:G-hyp2},  \\[2pt]
\indent iii) $\rho_0\in\P(\S)$ satisfies $\supp(\rho_0)\subset \overline D_r$.  \\[2pt]
Consider $\rho\in\Cont([0,\infty);\P(\S))$ the global weak solution to \eqref{eqn:model} starting from $\rho_0$, as provided by Proposition \ref{prop:inv-cont}. Then, there exists $P\in \overline D_r$ such that $W_1(\rho_t,\delta_P) \to 0$ as $t\to\infty$.
\end{thm}
\begin{proof}
	By Proposition \ref{prop:inv-cont}, we have 
\[ \supp(\rho_t) \subset \overline D_r \quad \mbox{for all $t\in[0,\infty)$} \]
and hence, from Prokhorov's theorem we infer the existence of $\rho_\infty\in\P(\S)$ such that $\supp(\rho_\infty) \subset \overline D_r$ and $(\rho_t)_{t\geq0}$ converges narrowly to $\rho_\infty$. As the rotation group is compact, we further get $W_1(\rho_t,\rho_\infty) \to 0$ as $t\to\infty$. We also note that the sequence of product measures $(\rho_t \otimes \rho_t)_{t\geq0}$ converges narrowly to $\rho_\infty\otimes\rho_\infty$. 
	
Suppose by contradiction that there exist $Q_1,Q_2\in\supp(\rho_\infty)$ with $Q_1\neq Q_2$. By assumption \eqref{eqn:G-hyp1} on $\G$, there exists $0<\bar{\delta}< d(Q_1,Q_2)/4$ such that $\G(\bar{\delta})>0$. Note that $B_{\bar{\delta}}(Q_1) \cap B_{\bar{\delta}}(Q_2) = \emptyset$ and hence, $(\rho_\infty\otimes\rho_\infty)(B_{\bar{\delta}}(Q_1)\times B_{\bar{\delta}}(Q_2))>0$. Also, for any $R_1 \in B_{\bar{\delta}}(Q_1)$ and $R_2 \in B_{\bar{\delta}}(Q_2)$, by triangle inequality one has:
\[
d(R_1,R_2) \geq d(Q_1,Q_2)- d(Q_1,R_1) - d(Q_2,R_2) > 4 \bar{\delta} - \bar{\delta} - \bar{\delta} = 2 \bar{\delta}.
\]

Then, by narrow convergence of $\rho_t \otimes \rho_t$ we have:
\begin{align*}
\begin{aligned}
&\lim_{t \to \infty}  \iint_{d(R_1,R_2)> 2 \bar{\delta}} \d\rho_t(R_1) \d\rho_t(R_2) =  \iint_{d(R_1,R_2)> 2 \bar{\delta}} \d\rho_\infty(R_1) \d\rho_\infty(R_2) \\
& \hspace{0.5cm} \geq \iint_{B_{\bar{\delta}}(Q_1)\times B_{\bar{\delta}}(Q_2)} \d\rho_\infty(R_1) \d\rho_\infty(R_2)  >0,
\end{aligned}
\end{align*}
which contradicts \eqref{eqn:limit}. We infer that $\supp(\rho_\infty)$ is a singleton, which concludes the proof.
\end{proof}


\subsection{Asymptotic consensus in the discrete model}
\label{subsect:consensus-disc}
We consider the specific case of the discrete model (see Section \ref{subsect:particle}), where solutions are empirical measures. The results in Section \ref{subsect:consensus-cont}, in particular Theorem \ref{thm:consensus-cont}, apply of course to such weak measure-valued solutions. However, using the discrete nature of the model we can establish separate asymptotic results, using different assumptions on the interaction function.  \newline

Consider the discrete model \eqref{eqn:model-disc} for $\N$ particles of identical masses $m_i = 1/\N$ on $\S$:\be\label{eqn:model-discrete}
	\begin{cases} 
	{\dot R}_i = - \displaystyle \frac1\N  \sum_{j=1}^{\N} \nabla K_{R_j}(R_i), \quad t > 0, \\ 
	R_i(0) = R_i^0.
	\end{cases}
\ee
The assumption on identical masses is made for convenience, as results extend immediately to general masses $m_i \in (0,1)$. The discrete analogue of the energy functional \eqref{eqn:energy-cont} is the function $E_\N\: \S^\N \to \R$ given by:
\begin{equation}
\label{eqn:energy-discrete}
	E_\N(x_1,\dots,x_\N)=\frac{1}{\N^2}\sum_{1\leq i\leq j\leq \N}K(R_i, R_j), \qquad\mbox{for all $(R_1,\dots,R_\N)\in \S^\N$}.
\end{equation}
One can write the dynamics in \eqref{eqn:model-discrete} as
\begin{equation}
\label{eqn:grad-flow}
	{\dot R}_i(t)=-\N\nabla_i E_\N(R_1(t),\dots,R_\N(t)),
\end{equation}
where $\nabla_i$ stands for the manifold gradient with respect to the $i$-th variable. The discrete model \eqref{eqn:model-discrete} is a gradient flow with respect to the discrete energy $E_\N$. The following lemma will be used in the first results on discrete consensus.
\begin{lem}
\label{lem:technical} Let $R_1, \dots, R_{\N}\in D_{\pi/4}$ be such that $\Delta:=\max_{1\leq i,j\leq \N}d(R_i,R_j) >0$ and  if necessary by relabeling, we may assume that $d(R_1,R_2) = \Delta$. Then, one has
\[
\log_{R_1}R_2 \cdot  \log_{R_1}R_j \geq 0, \qquad \mbox{for all $j\in\{1,\dots,\N\}$}.
\]
\end{lem}
\begin{proof}
By definition of $\Delta$ and the fact that $d(R_1,R_2) = \Delta$, one has $R_j \in \overline{D}_\Delta(R_2)$ for all $j\in\{1,\dots,\N\}$. Suppose that $\N\geq3$, as the case $\N=2$ is trivial. Fix $j \in\{3,\dots,\N\}$ and consider the minimizing geodesic $R:[0,1] \to SO(3)$ between $R_1$ and $R_j$, parametrized so that $R(0)=R_1$ and $R'(0)=\log_{R_1}R_j$. Then, by the chain rule and \eqref{eqn:gradd} we find
\be\label{eqn:dy2}
	\frac{\der}{\der t} d(R(t), R_2)^2 = \grad d_{R_2}^2(R(t)) \cdot {\dot R}(t) =-2 \log_{R(t)}R_2 \cdot {\dot R}(t).
\ee
Since $\Delta < \pi/2$, the closed disk $\overline{D}_\Delta(R_2)$ is geodesically convex. Consequently, $R(t) \in \overline{D}_\Delta(R_2)$ and $d(R(t),R_2) \leq d(R_1,R_2)$ for all $t\in[0,1]$. Furthermore, the map $t\mapsto d(R(t),R_2)^2$ is nonincreasing at $t=0$, which together with \eqref{eqn:dy2} implies:
\begin{equation*}
0 \geq \left.\frac{\der}{\der t}\right|_{t=0} d(R(t), R_2)^2 =-2 \log_{R_1}R_2 \cdot \log_{R_1}R_j.
\end{equation*}
The conclusion then follows.
\end{proof}

The following theorem is the first of two results on the asymptotic consensus for the intrinsic discrete model on $SO(3)$. 
\begin{thm}
\emph{(Asymptotic consensus I)}
\label{thm:consensus-d1}
Let $K$ satisfy \ref{hyp:K} and $r<\pi/4$ be a fixed positive number. Suppose that  \\[2pt]
\indent i) $g'$ is continuously differentiable on $[0,4r^2]$ and satisfies $g'(s) \geq \c s^\alpha$ for all $s \in [0,4r^2]$, for some $\c>0$ and $\alpha \geq 0$,  \\[2pt]
\indent ii) initial points $\{ R_i^0 \}_{i=1}^\N$ satisfy $(R_i^0)_{i=1}^\N \subset \overline D_r$, \\[2pt]
and let $(R_i(t))_{i=1}^\N$ be the global solution to \eqref{eqn:model-discrete} whose well-posedness is guaranteed by Proposition \ref{prop:inv-disc}. Then, 
\[ \lim_{t \to \infty} d(R_i(t), R_j(t)) = 0, \qquad \mbox{for all $i,j\in\{1,\dots \N\}$}. \]
\end{thm} 
\begin{proof}
By abuse of notation, we denote $E_\N(t) = E_\N(x_1(t),\dots,x_\N(t))$, and  set up the empirical measures \eqref{eq:atomic-initial} and \eqref{eq:atomic}. Then, note that 
\[ E_\N(t) = E[\rho_t^\N], \]
where $E$ is the continuum energy \eqref{eqn:energy-cont}. By Lemma \ref{lem:E-lim}, $E_\N(t) \to E_\infty$ as $t\to\infty$ for some $E_\infty\in\R$, and also, $\ddot{E}_\N(t)$ is bounded on $[0,\infty)$. Hence, we apply Barbalat's lemma for $t\mapsto E_N(t)$ to get
\[
	\dot{E}_\N(t) \rightarrow 0, \qquad \text{ as  $t \to \infty$}.
\]
Using \eqref{eqn:grad-flow}, we can see for all $t\in[0,\infty)$,
\begin{equation}
	\dot{E}_\N(t)=\sum_{i=1}^{\N} \nabla_i E_\N(R_1(t),\dots,R_\N(t)) \cdot {\dot R}_i(t) = -\frac{1}{\N}\sum_{i=1}^{\N} |{\dot R}_i(t)|^2.
\end{equation}
This implies 
\begin{equation}
\label{eqn:equil}
{\dot R}_i(t)\to 0,  \qquad \text{ as $t \to \infty$, for all $i\in\{1,\dots,\N\}$}.
\end{equation}

Let $\Delta\: [0,\infty) \to [0,\infty)$ be defined by
\begin{equation}
\label{eqn:Delta}
	\Delta(t) := \max_{1\leq i,j\leq \N}d(R_i(t),R_j(t)), \qquad \mbox{for all $t\in[0,\infty)$}.
\end{equation}
If necessary, by relabeling particles we may assume that 
\[
d(R_1(t),R_2(t)) = \Delta(t), \qquad \mbox{for all $t\in[0,\infty)$}.
\]
Now, we claim:
\[  \lim_{t \to \infty} \Delta(t) = 0. \] 
Since the initial data is supported in $\overline D_r$, by Proposition \ref{prop:inv-disc} we also have 
\[ R_i(t)\in \overline D_r \quad \forall~t\in[0,\infty), \quad 1 \leq i \leq N, \]
where $r<\pi/4$. In particular, we can apply Lemma \ref{lem:technical} to $(R_i(t))_{i=1}^\N$. Now we take the inner product with $\log_{R_1(t)}R_2(t)$ on both sides of \eqref{eqn:model-discrete} for particle $i=1$ to get
\begin{align}
\begin{aligned} \label{est:dp}
{\dot R}_1 \cdot \log_{R_1}R_2 &= \frac{1}{\N}\sum_{j=1}^{\N} 2g'(d(R_1, R_j)^2)\log_{R_1}R_j \cdot \log_{R_1}R_2 \\
&\geq \frac{2}{\N} g'(d(R_1, R_2)^2) | \log_{R_1}R_2 |^2, 
\end{aligned}
\end{align}
where for the inequality we used Lemma \ref{lem:technical} to bound from below a sum of nonnegative terms by the term with $j=2$. By the Cauchy--Schwarz inequality $|{\dot R}_1 \cdot \log_{R_1}R_2 | \leq |R_1'| |\log_{R_1}R_2|$ and $|\log_{R_1}R_2| = d(R_1,R_2)$,  one have
\[
\frac{2}{\N} g'(d(R_1, R_2)^2) d(R_1,R_2) \leq |{\dot R}_1|.
\]
Now, using the assumption on $g'$ we get
\[
\frac{2}{\N} \c \, d(R_1,R_2)^{1+2\alpha} \leq |{\dot R}_1|,
\]
where $1+2\alpha >0$. 

Since ${\dot R}_1$ approaches to $0$ as $t \to \infty$ by \eqref{eqn:equil},  so does $d(R_1(t),R_2(t))$. This completes the proof. 
\end{proof}
Next, we move to the second result on discrete consensus. Under a stricter assumption on the interaction potential ($g'$ bounded from below by a positive constant) we can show that the asymptotic consensus emerges exponentially fast. Using notation \eqref{eqn:G}, the discrete model \eqref{eqn:model-discrete} can be written as:
\begin{equation}
\label{eqn:model-disc2}
{\dot R}_i =\frac{1}{\N}\sum_{k=1}^\N\G(d(R_i,R_k))(R_k-R_iR_k^TR_i), \qquad i=1,\dots,\N.
\end{equation}
For simplicity, we set 
\[ \G_{ik}:=\G(d(R_i,R_k)), \quad \mbox{for arbitrary indices $i$ and $k$}. \]
For two fixed indices $i$ and $j$ we then have:
\begin{align}
\frac{d}{d t}(R_i^TR_j)&=\dot{R}_i^TR_j+R_i^T\dot{R}_j  \\
&=\frac{1}{\N}\sum_{k=1}^\N\left(\G_{ik}(R_k^TR_j-R_i^TR_kR_i^TR_j)+\G_{jk}(R_i^TR_k-R_i^TR_jR_k^TR_j)\right). \label{eqn:dt-prod}
\end{align}
By simple manipulations, one has
\begin{align*}
\begin{aligned}
&R_k^TR_j-R_i^TR_kR_i^TR_j \\
& \hspace{0.5cm} =2R_k^TR_j-(I+R_i^TR_kR_i^TR_k)R_k^TR_j\\
& \hspace{0.5cm} =2R_k^TR_j-2R_i^TR_j-(I-2R_i^TR_k+R_i^TR_kR_i^TR_k)R_k^TR_j\\
& \hspace{0.5cm} =2(R_k-R_i)^TR_j -(I-R_i^TR_k)^2R_k^TR_j.
\end{aligned}
\end{align*}
We apply trace to the above relation to find
\begin{align}
&\mathrm{tr}(R_k^TR_j-R_i^TR_kR_i^TR_j) \\
& \hspace{0.5cm} =2\mathrm{tr}((R_k-R_i)^TR_j)-\mathrm{tr}((I-R_i^TR_k)^2R_k^TR_j)  \\
& \hspace{0.5cm} =2\mathrm{tr}((R_k-R_i)^TR_j)-\mathrm{tr}(R_i^T(R_i-R_k)(R_k^T-R_i^T)R_j) \\
& \hspace{0.5cm} =2\mathrm{tr}((R_k-R_i)^TR_j)-\mathrm{tr}((R_i-R_k)(R_k^T-R_i^T)R_jR_i^T). 
 \label{eqn:dcons-calc1}
\end{align}
We calculate further the second term on the right-hand-side of \eqref{eqn:dcons-calc1}. \newline

Note that 
\begin{align*}
\begin{aligned}
&\mathrm{tr}((R_i-R_k)(R_k^T-R_i^T)R_jR_i^T) = \mathrm{tr}(((R_i-R_k)(R_k^T-R_i^T)R_jR_i^T)^T) \\
& \hspace{0.5cm} =\mathrm{tr}(R_iR_j^T(R_k-R_i)(R_i-R_k)^T) =\mathrm{tr}((R_k-R_i)(R_i-R_k)^TR_iR_j^T).
\end{aligned}
\end{align*}
Therefore we have
\begin{align*}
\begin{aligned}
&2\mathrm{tr}((R_i-R_k)(R_k^T-R_i^T)R_jR_i^T) \\ 
& \hspace{0.5cm} =\mathrm{tr}((R_i-R_k)(R_k-R_i)^T(R_jR_i^T+R_iR_j^T))\\
& \hspace{0.5cm} =-\mathrm{tr}((R_i-R_k)(R_i-R_k)^T(2I-(R_i-R_j)(R_i-R_j)^T))\\
& \hspace{0.5cm} =-2\mathrm{tr}((R_i-R_k)(R_i-R_k)^T)+\mathrm{tr}((R_i-R_k)(R_i-R_k)^T(R_i-R_j)(R_i-R_j)^T).
\end{aligned}
\end{align*}
Now, we return to \eqref{eqn:dcons-calc1} and find
\begin{align}
&\mathrm{tr}(R_k^TR_j-R_i^TR_kR_i^TR_j) \\
& \hspace{1cm} =2\mathrm{tr}((R_k-R_i)^TR_j) +\mathrm{tr}((R_i-R_k)(R_i-R_k)^T) \\
& \hspace{1cm} -\frac{1}{2}\mathrm{tr}((R_i-R_k)(R_i-R_k)^T(R_i-R_j)(R_i-R_j)^T)  \\
& \hspace{1cm} =2\mathrm{tr}((R_k-R_i)^TR_j)+{\|R_i-R_k\|}_F^2 -\frac{1}{2} {\|(R_i-R_k)^T(R_i-R_j)\|}_F^2.
 \label{E-1}
\end{align}
By Lemma \ref{L1} we have:
\begin{equation} \label{NNN-2}
{\|(R_i-R_k)^T(R_i-R_j)\|}_F^2\leq {\|R_i-R_k\|}_F^2 \cdot {\|R_i-R_j\|}_F^2.
\end{equation}
Next, we use \eqref{NNN-2} and \eqref{E-1} to find
\begin{align}
&\mathrm{tr}(R_k^TR_j-R_i^TR_kR_i^TR_j) \\
& \hspace{0.5cm} \geq2\mathrm{tr}((R_k-R_i)^TR_j)+{\|R_i-R_k\|}_F^2-\frac{1}{2}{\|R_i-R_k\|}_F^2 {\|R_i-R_j\|}_F^2 \\
&  \hspace{0.5cm} =2\left(\mathrm{tr}(R_k^TR_j)-\mathrm{tr}(R_i^TR_j)\right)+\|R_i-R_k\|_F^2\left(1-\frac{1}{2} {\|R_i-R_j\|}_F^2\right)  \\
& \hspace{0.5cm} =2\left(\mathrm{tr}(R_k^TR_j)-\mathrm{tr}(R_i^TR_j)\right)+{\|R_i-R_k\|}_F^2\left(\mathrm{tr}(R_i^TR_j)-2\right),
\label{dcons:ineq1}
\end{align}
where for the last equality we used \eqref{eqn:Fdist-ip}. \newline

Similarly, by interchanging indices $i$ and $j$, one can also obtain
\begin{equation}
\label{dcons:ineq2}
\mathrm{tr}(R_i^TR_k-R_i^TR_jR_k^TR_j)\geq 2(\mathrm{tr}(R_k^TR_i)-\mathrm{tr}(R_i^TR_j))+\|R_j-R_k\|_F^2\left(\mathrm{tr}(R_i^TR_j)-2\right).
\end{equation}

Now apply trace to \eqref{eqn:dt-prod} and use \eqref{dcons:ineq1} and \eqref{dcons:ineq2} to get
\begin{align}
\begin{aligned}  \label{dcons:ineq-total}
\frac{d}{dt}\mathrm{tr}(R_i^TR_j) &\geq\frac{2}{\N}\sum_{k=1}^N\left( \G_{ik}(\mathrm{tr}(R_k^TR_j)-\mathrm{tr}(R_i^TR_j))+\G_{jk}(\mathrm{tr}(R_k^TR_i)-\mathrm{tr}(R_i^TR_j) \right)   \\
&\quad +\frac{1}{\N} \left(\mathrm{tr}(R_i^TR_j)-2 \right)\sum_{k=1}^N \left(\G_{ik}{\|R_i-R_k\|}_F^2+\G_{jk}{\|R_j-R_k\|}_F^2\right).
\end{aligned}
\end{align}

In the following theorem, we use the inequality \eqref{dcons:ineq-total} to derive exponential consensus.
\begin{thm}
\emph{(Asymptotic consensus II)}
\label{thm:consensus-d2} Let $K$ satisfy \ref{hyp:K} and $r<\pi/6$ be a fixed positive number. Suppose that  \\[2pt]
\indent i) $g'(s) \geq \c$ for all $s \in [0,4r^2]$, for some $\c>0$,  \\[2pt]
\indent ii) initial points $\{ R_i^0 \}_{i=1}^\N$ satisfy $(R_i^0)_{i=1}^\N \subset \overline D_r$,  \\[2pt]
and let $(R_i(t))_{i=1}^\N$ be the global solution to \eqref{eqn:model-discrete} (or equivalently, \eqref{eqn:model-disc2}) whose well-posedness is guaranteed by Proposition \ref{prop:inv-disc}. Then, one has
\[ \lim_{t \to \infty} d(R_i(t), R_j(t)) = 0, \quad \mbox{exponentially fast}~~\mbox{for all $i,j\in\{1,\dots \N\}$}. \]
\end{thm}
\begin{proof} We set 
\[
\mathcal{T}_m(t)=\min_{i, j}\mathrm{tr}\left(R_i^T(t)R_j(t)\right).
\]
Since $\{R_i(t)\}_{i=1}^{\N}$ is an analytic solution, for fixed time $T>0$ there exists
\[
0=t_0<t_1<\cdots<t_M=T,
\]
and
\[
i_m, j_m\in\{1, 2, \cdots, \N\} \quad \text{ for all } m\in\{1, 2, \cdots, M\},
\]
such that
\[
\mathcal{T}_m(t)=\mathrm{tr}\left(R_{i_m}^T(t)R_{j_m}(t)\right) \quad \text{ for all }  t\in[t_{m-1},t_m].
\]
For fixed $m\in\{1, 2, \cdots, M\}$, it follows from \eqref{dcons:ineq-total} and the definition of $i_m$ and $j_m$ that we have the following inequality for  $t\in (t_{m-1},t_m)$:
\begin{align*}
\begin{aligned}
& \frac{d}{dt}\mathrm{tr}(R_{i_m}^TR_{j_m}) \\
& \hspace{0.5cm} \geq\frac{2}{\N}\sum_{k=1}^N\left( \G_{i_mk}(\mathrm{tr}(R_k^TR_{j_m})-\mathrm{tr}(R_{i_m}^TR_{j_m}))+\G_{j_mk}(\mathrm{tr}(R_k^TR_{i_m})-\mathrm{tr}(R_{i_m}^TR_{j_m}) \right)\\
& \hspace{0.7cm}+\frac{1}{\N}\left(\mathrm{tr}(R_{i_m}^TR_{j_m})-2\right)\sum_{k=1}^N \left(\G_{i_mk} {\|R_{i_m}-R_k\|}_F^2+\G_{j_mk}{\|R_{j_m}-R_k\|}_F^2 \right)\\
&  \hspace{0.5cm} \geq \frac{1}{\N}\left(\mathrm{tr}(R_{i_m}^TR_{j_m})-2\right)\sum_{k=1}^N \left(\G_{i_mk} {\|R_{i_m}-R_k\|}_F^2+\G_{j_mk}{\|R_{j_m}-R_k\|}_F^2 \right)\\
&  \hspace{0.5cm} \geq\frac{1}{\N} \left(\mathrm{tr}(R_{i_m}^TR_{j_m})-2\right) \left(\G_{i_m j_m}{\|R_{i_m}-R_{j_m}\|}_F^2+\G_{j_m i_m}
{\|R_{j_m}-R_{i_m}\|}_F^2 \right).
\end{aligned}
\end{align*}
Finally, since by symmetry $\G_{i_m j_m} = \G_{j_m i_m}$, we have
\begin{align}
\frac{d}{dt}\mathrm{tr}(R_{i_m}^TR_{j_m}) &\geq \frac{2}{\N} \G_{i_mj_m} \left(\mathrm{tr}(R_{i_m}^TR_{j_m})-2 \right){\|R_{i_m}-R_{j_m}\|}_F^2 \\
& =\frac{2}{\N} \G_{i_mj_m} \left(\mathrm{tr}(R_{i_m}^TR_{j_m})-2\right) \left(6-2\mathrm{tr}(R_{i_m}^TR_{j_m}) \right),
 \label{E-2}
\end{align}
for all $t\in (t_{m-1},t_m)$. \newline

Now note that at the initial time any two particles $R_i^0,R_j^0$ satisfy $d(R_i^0,R_j^0) \leq 2r <\frac{\pi}{3}$, which by \eqref{eqn:geod-dist} it implies that $\mathrm{tr}((R_i^0)^TR_j^0)>2$. Also, by the Cauchy-Schwartz inequality, 
\[ \mathrm{tr}((R_i^0)^TR_j^0) \leq {\| R_i^0 \|}_F {\| R_j^0 \|}_F = 3. \]
Consequently, $\mathcal{T}_m(0) \in (2,3]$.  By the assumption on $g'$, one has 
\[ \G_{i_m j_m} \geq \c, \]
and hence, it follows from \eqref{E-2} that
\[
\frac{d}{dt}\mathcal{T}_m(t)\geq -\frac{4 c}{\N}(\mathcal{T}_m(t)-2)(\mathcal{T}_m(t)-3),
\]
for as long as $\mathcal{T}_m(t) \in (2,3]$. The differential inequality can be integrated to get:
\begin{equation}
\label{ineq:E-3}
\mathcal{T}_m(t)\geq 3+\frac{(\mathcal{T}_m(0)-3)\exp\left(-\frac{4 \c}{\N}t\right)}{(\mathcal{T}_m(0)-2)-(\mathcal{T}_m(0)-3)\exp\left(-\frac{4 \c}{\N}t\right)}.
\end{equation}
Finally, we combine the relation $\mathcal{T}_m(t)\leq 3$ (by the Cauchy-Schwartz inequality) and \eqref{ineq:E-3}, one has 
\[ \mathcal{T}_m(t) \in (2,3] \quad \mbox{for all $t \in [0,t_1)$.} \] 
The procedure can be continued to find that \eqref{ineq:E-3} holds for all $t \in [0,T]$, for arbitrary $T>0$. In particular, we derive that $\mathcal{T}_m(t)$ converges exponentially fast to $3$ as $t\to \infty$, which by \eqref{eqn:geod-dist}, it is equivalent to $\min_{i,j} \cos(d(R_i(t),R_j(t)))\to 0$ exponentially fast as $t \to \infty$. Consequently,  $\max_{i,j} d(R_i(t),R_j(t))\to 0$ exponentially fast as $t \to \infty$.
\end{proof}

\vspace{0.5cm}

Next, we present some examples of interaction potentials  satisfying the assumptions in Theorems \ref{thm:consensus-cont}, \ref{thm:consensus-d1} and \ref{thm:consensus-d2}. 

\paragraph{Examples.} 
\begin{enumerate}
\item {\em (Power-law potentials).} Consider the purely attractive power-law potential with exponent $q$:
\[
K(R, Q)=\frac{1}{q}d(R, Q)^{q}, \qquad \text { for } g(s)=s^{\frac{q}{2}}.
\]
This potential satisfies the assumptions of Theorem \ref{thm:consensus-cont} (see Remark \ref{rem:gG}) and Theorem  \ref{thm:consensus-d1} for $q=2$ (quadratic potential) and $q\geq 4$. For example, in Theorem \ref{thm:consensus-d1} one has indeed that $g$ is of class $C^2$ and that $g'(s) \geq \c s^\alpha$ for all $s\in[0,\infty)$, with $\c=\frac{q}{2}$ and $\alpha = \frac{q}{2}-1$. The assumptions of Theorem  \ref{thm:consensus-d2} are satisfied only for $q=2$, the quadratic potential. 

In Section \ref{sect:numerics}, we will present some numerical experiments with power-law potentials for the aggregation model on $SO(3)$.

\item {\em (Potential in the Lohe sphere model).} The following potential
\[
K(R,Q)=2\sin^2\left(\frac{{d(R,Q)}}{2}\right), \qquad \text{ for } g(s)=2 \sin^2\left(\frac{\sqrt{s}}{2}\right),
\]
was considered in recent works \cite{HaKoRy2018,HaKiLeNo2019} for the Lohe sphere model. By a direct calculation,
\[
g'(s)=\frac{1}{2} \frac{\sin\frac{\sqrt{s}}{2}}{\frac{\sqrt{s}}{2}}\cos\frac{\sqrt{s}}{2},
\]
and hence, $\G(s)=\frac{1}{2}$ for all $s>0$. One can also check that $g$ is of class $C^2$.

The assumptions of Theorem \ref{thm:consensus-cont} are trivially satisfied, since $\G$ is a constant function. Also, for $0<r<\pi/4$, it holds that
\[
g'(s)\geq \frac{\cos r}{2}, \qquad \mbox{for all $s \in [0,4r^2]$},
\]
so $g'$ satisfies the lower bound assumptions in both Theorems \ref{thm:consensus-d1} and \ref{thm:consensus-d2} (for the former take $\c= \cos(r) /2$ and $\alpha =0$).
\end{enumerate}


\section{Numerical results}
\label{sect:numerics}

We present some numerical experiments for the discrete model \eqref{eqn:model-discrete}, which we solve numerically using the angle-axis representation. Specifically, we write \eqref{eqn:model-discrete}  as an ODE system for the angle-axis pairs $(\theta_i,\bfv_i )$, where $\theta_i \hat \bfv_i = \log R_i$, $i=1, \dots, N$, and solve it numerically with the 4th order Runge-Kutta method. In all simulations we have initialized $\theta_i$ randomly in the interval $(0,\pi/4)$, while the unit vectors $\bfv_i$ were generated in spherical coordinates, with the polar and azimuthal angles drawn randomly in the intervals $(0,\pi)$ and $(0,2 \pi)$, respectively. Consequently, all rotation matrices $R_i$ at time $t=0$ are within distance $\pi/4$ from the identity matrix and hence, satisfy the assumptions of Theorem \ref{thm:consensus-d1}.

For plotting, we identify $SO(3)$ with a ball in $\R^3$ of radius $\pi$ centred at the origin. The identity matrix $I$ corresponds to the centre of the ball, while an arbitrary point within this ball represents a rotation matrix, with rotation angle given by the distance from the point to the centre, and axis given by the ray from the centre to the point. For a correct representation, antipodal points on the surface of the ball have to be identified, as they represent the same rotation matrix (rotation by $\pi$ about a ray gives the same result as rotation by $\pi$ about the opposite ray).

We present numerical experiments with two types of potentials, power-law and Morse-type, both considered in the context of intrinsic interactions.  A general power-law potential reads:
\begin{equation}
\label{eqn:K-pl}
 K(R,Q) = -\frac{1}{p}d(R,Q)^p + \frac{1}{q}d(R,Q)^q,
\end{equation}
where the exponents $p$ and $q$ (with $p<q$) correspond to repulsive and attractive interactions, respectively. The case of purely attractive power-law potentials was discussed in Section \ref{subsect:consensus-disc} (see Example 1). A (generalized) Morse-type potentials \cite{CaHuMa2014} is given by:
\begin{equation}
\label{eqn:K-gM}
K(R,Q)  = V(d(R,Q)) - CV (d(R,Q)/l),
\end{equation}
where
\begin{equation}
\label{eqn:Morse}
V(r) = -e^{-\frac{r^s}{s}}, \qquad \text{ with } s > 0,
\end{equation}
and $C$, $l$ are positive constants, which control the relative size and range of the repulsive interactions. 

Both power-law and Morse-type potentials have been widely used in the aggregation literature for Euclidean spaces \cite{Balague_etalARMA, BaCaLaRa2013, ChFeTo2015, FeHuKo11, FeHu13}. In particular, Morse-type potentials can enable explicit calculations of equilibrium solutions \cite{BeTo2011,CaHuMa2014}. It was shown that such potentials can lead, through a delicate balance between attraction and repulsion, to a diverse set of equilibria, such as aggregations on disks, annuli, rings and delta concentrations \cite{KoSuUmBe2011,FeHuKo11,CaHuMa2014}. Potentials of form \eqref{eqn:K-gM}-\eqref{eqn:Morse} have been also used in other models for swarming and flocking \cite{Chuang_etal}. 

In Figure \ref{fig:consensus} we show results for several simulations using $\N=20$ particles and purely attractive power-law potentials (potential in the form \eqref{eqn:K-pl}, but with no repulsion term). The plots in Figure \ref{fig:consensus}(a) and (b) correspond to an attractive quadratic potential ($q=2$). Initial particles located at black dots in Figure \ref{fig:consensus}(a) achieve asymptotic consensus at the point indicated by red diamond. The initial values of $\theta_i$ and their asymptotic state are shown in Figure \ref{fig:consensus}(b) by black and red circles, respectively. Note that for visualization purposes we do not show the full ball of radius $\pi$ in Figure \ref{fig:consensus}(a), as we set the axis limits to $[-1,1]$.  Figure \ref{fig:consensus}(c) illustrates the speed of convergence to consensus. It shows the evolution in time of the diameter $\Delta$ of the discrete set $\{R_i\}_{i=1}^N$ (see \eqref{eqn:Delta}) for different exponents $q$ of the attractive potential. 
Note that the larger the value of the exponent is, the slower the convergence to consensus. We also point out the exponential convergence for $q=2$, see Theorem \ref{thm:consensus-d2}. 

Figure \ref{fig:patterns} shows two equilibria of the discrete model obtained by running simulations with $\N=40$ particles to steady state.  Figure \ref{fig:patterns}(a), corresponding to a power-law potential with $p=2$ and $q=10$, shows an aggregation at four points. The distances to the identity matrix of these points are $0.5505$, $0.6002$, $0.6580$, $0.6678$, respectively. For Figure \ref{fig:patterns}(b) we used a Morse-type potential with $C=0.5$, $l=0.25$ $s=2$. The equilibrium locations (indicated by black dots) appear to lie on a geodesic sphere centred at the point indicated by red diamond. To find the centre we calculated numerically the Riemannian centre of mass of the equilibrium configuration by the intrinsic gradient descent algorithm investigated in \cite{AfsariTronVidal2013} (recall that the Riemannian centre of mass of a set of points on a manifold minimizes the sum of squares of the geodesic distances to the data points). We found the Riemannian centre of mass located at $\theta_{_C}=0.0958$, $\hat{v}_{_C}=(-0.8654,0.0701,0.4961)$.
The average distance of the equilibrium points to the centre of mass is $R =  0.3619$, with a standard deviation of $1.9 \times 10^{-4}$. Note that axis limits in Figure \ref{fig:patterns}(b) are set at $[-1,1]$, so we do not show the entire ball of radius $\pi$ there. Qualitatively similar equilibria were obtained with other parameter values as well.

\begin{figure}[!htbp]
 \begin{center}
 \begin{tabular}{ccc}
 \includegraphics[width=0.35\textwidth]{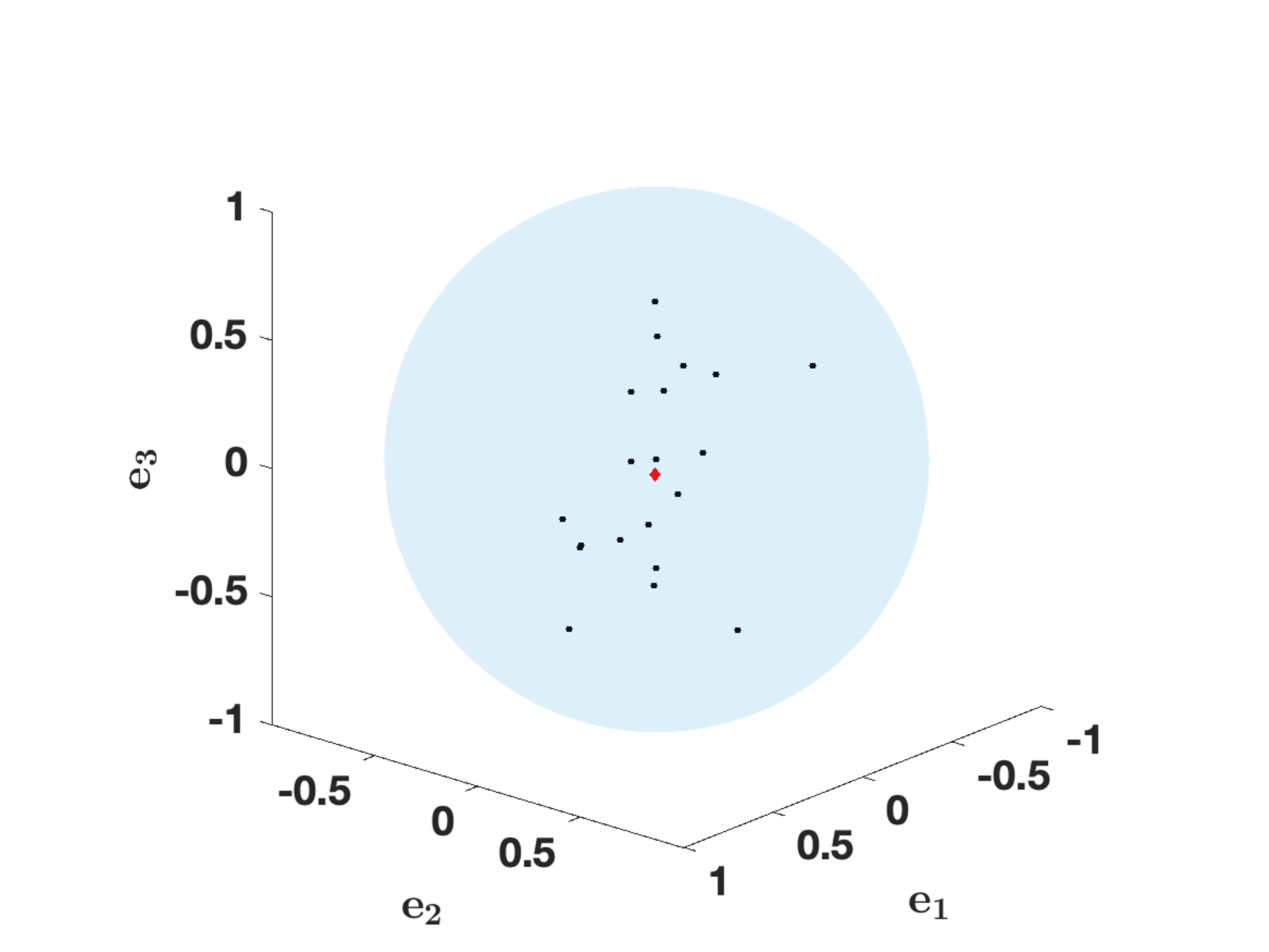} & 
 \includegraphics[width=0.3\textwidth]{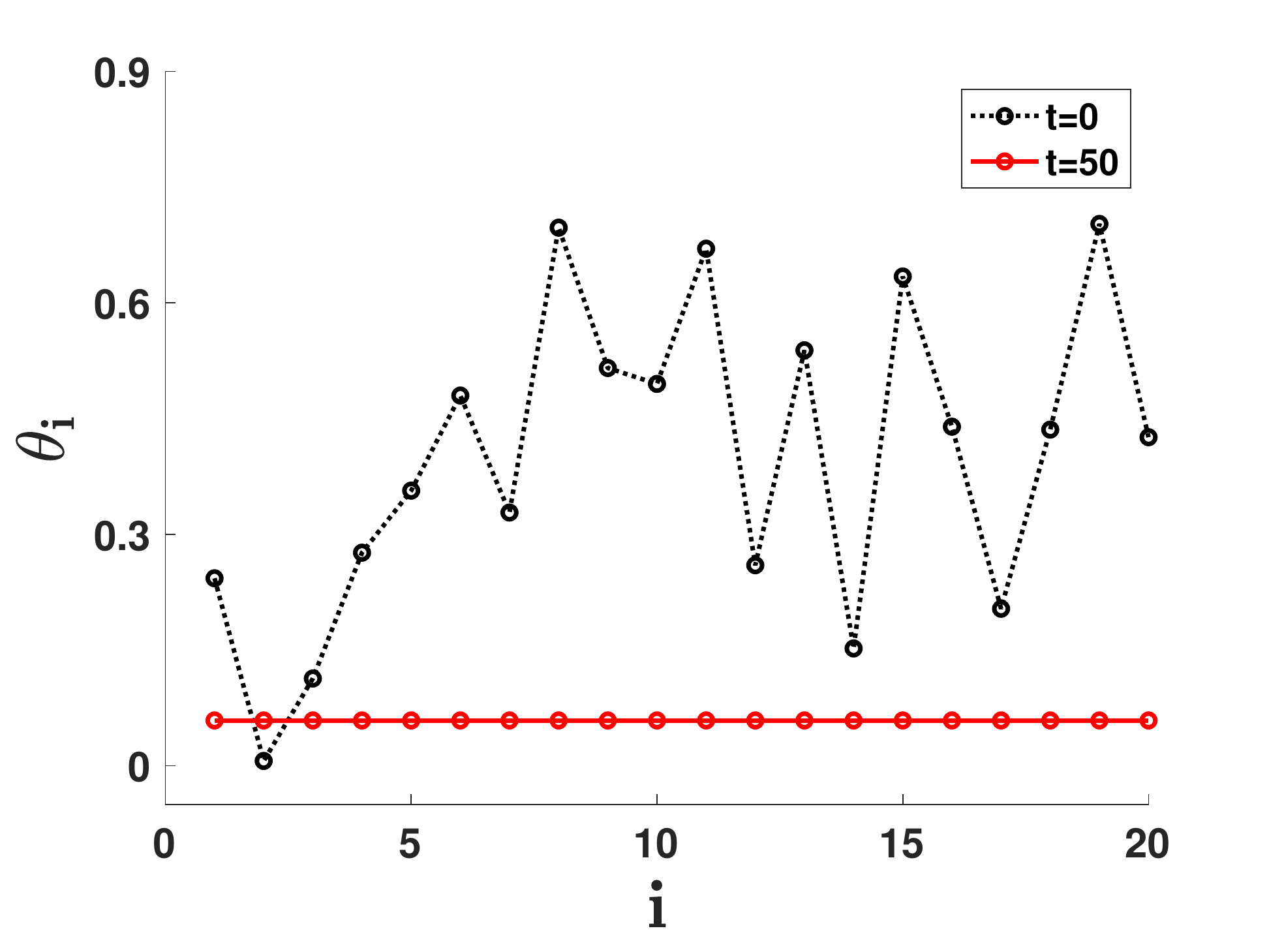}  &
 \includegraphics[width=0.3\textwidth]{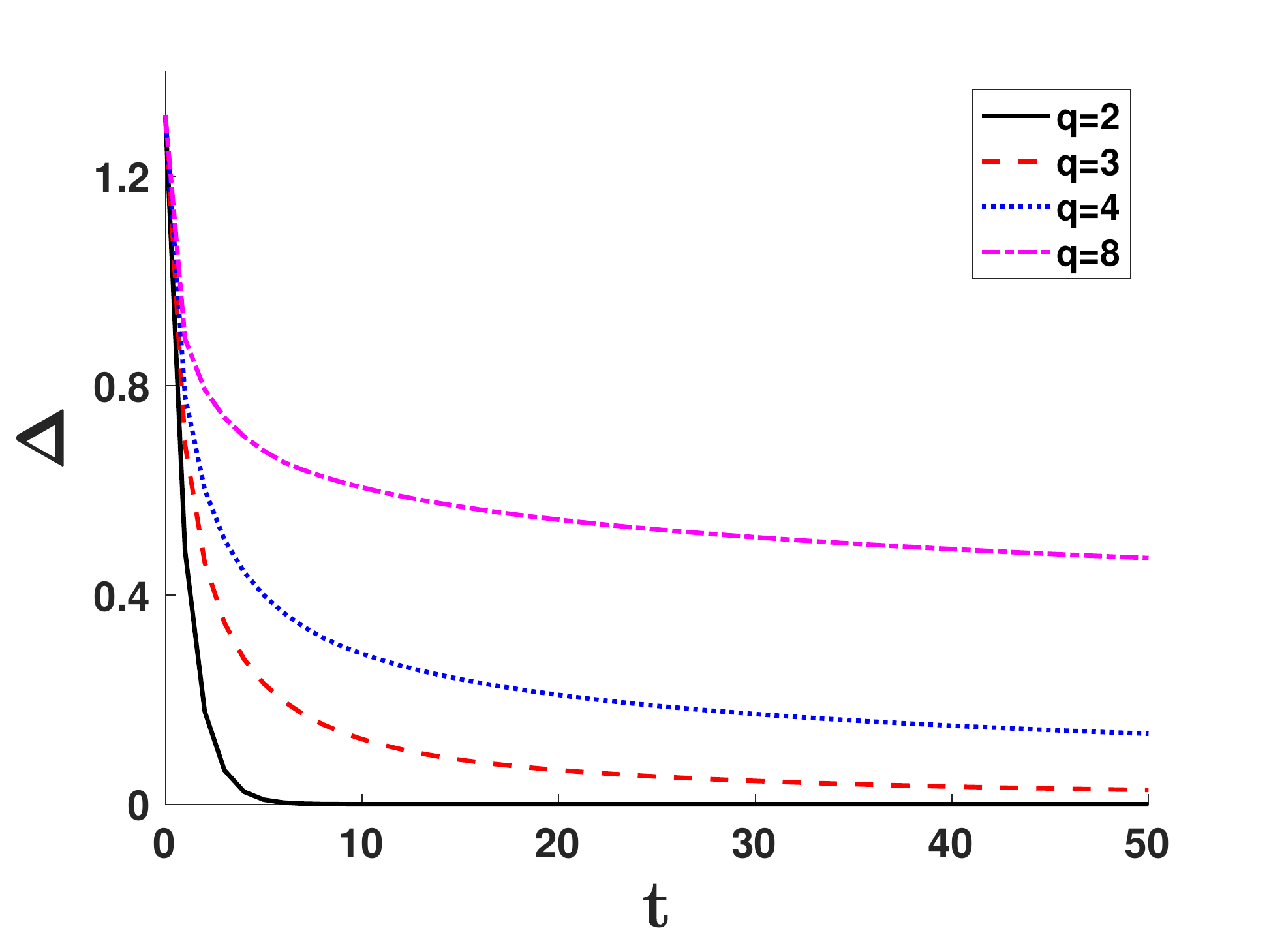} \\
 (a) & (b) & (c)  \end{tabular}
 \begin{center}
 \end{center}
\caption{Asymptotic consensus of $\N=20$ particles with attractive power-law potentials. (a) Quadratic potential ($q=2$). Initial particles are indicated by black dots, and the consensus location is shown with red diamond. (b) Same simulation as in (a), showing the initial values of $\theta_i$ and their asymptotic state (black and red circles, respectively). (c) Convergence to consensus for different values of $q$. Note the slower convergence with increasing $q$. The convergence is exponential for $q=2$ -- see Theorem \ref{thm:consensus-d2}.}
\label{fig:consensus}
\end{center}
\end{figure}

\begin{figure}[!htbp]
 \begin{center}
 \begin{tabular}{cc}
 \includegraphics[width=0.49\textwidth]{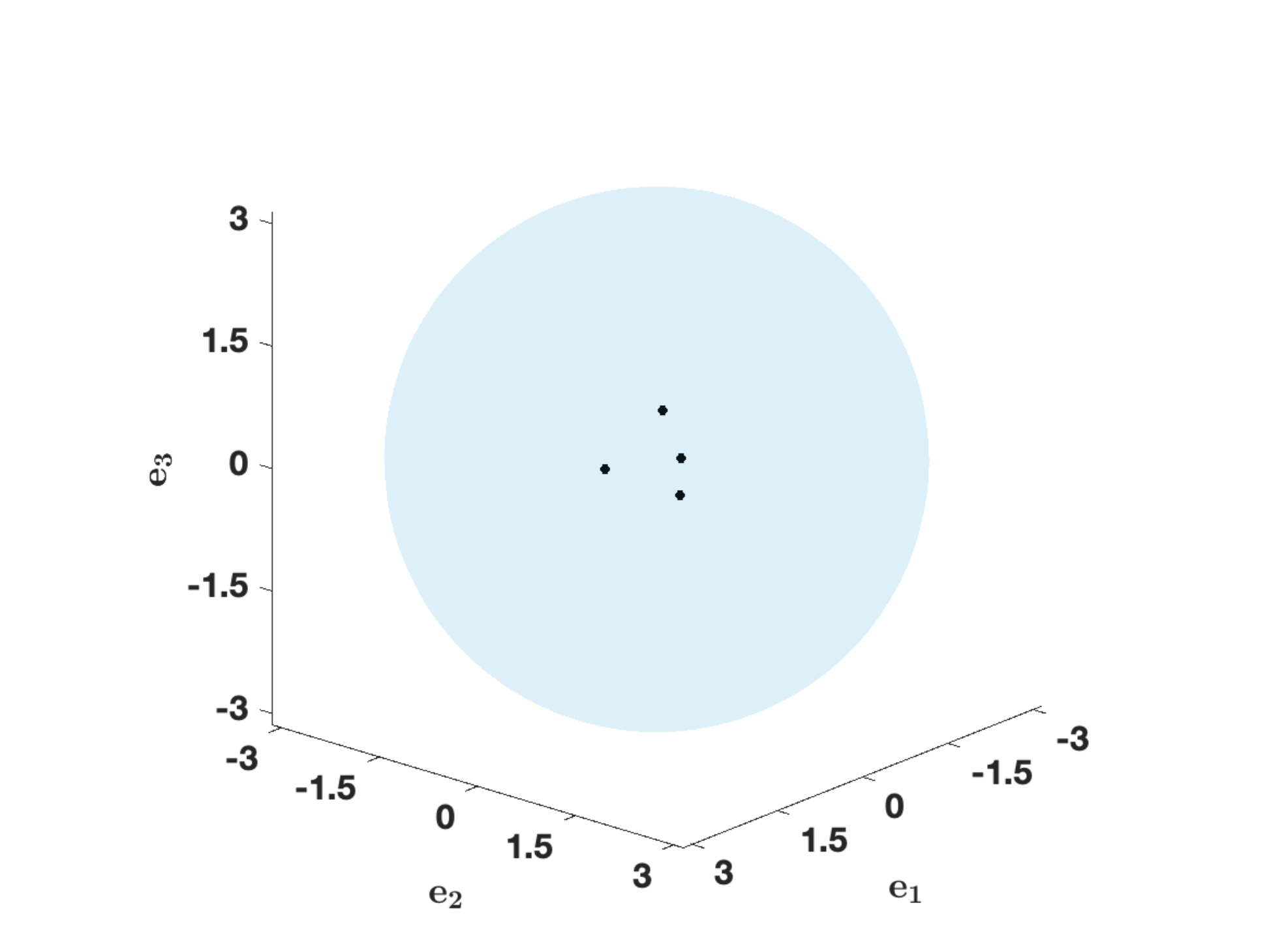} & 
 \includegraphics[width=0.45\textwidth]{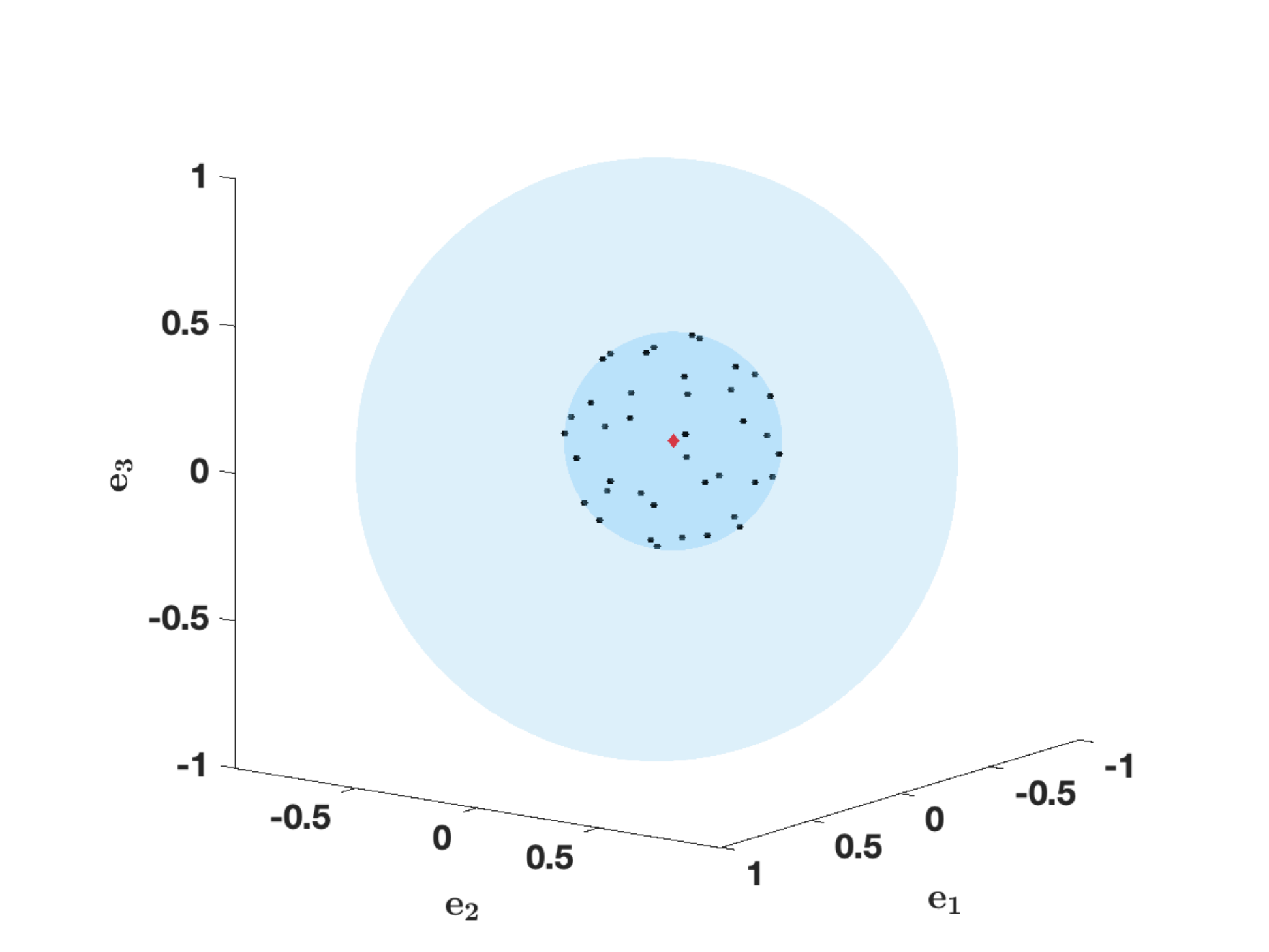}  \\
 (a) & (b) 
 \end{tabular}
 \begin{center}
 \end{center}
\caption{Equilibrium solutions of the discrete model. (a) Equilibrium aggregation at four points. The simulation corresponds to a power-law potential with $p=2$ and $q=10$. (b) Equilibrium solution (black dots) lying on a geodesic sphere in $SO(3)$. The simulation used a Morse-type potential with $C=0.1$, $l=0.2$, and $s=2$. The centre of the sphere (indicated by a red diamond) is the numerical Riemannian centre of mass of the equilibrium configuration. The average distance of the equilibrium points to the centre of mass is $R =  0.3619$, with a standard deviation of $1.9 \times 10^{-4}$.}
\label{fig:patterns}
\end{center}
\end{figure}

The numerical experiments presented here offer only a glimpse on the possible equilibria that can be obtained with the intrinsic model investigated in this paper. We expect the model to capture a rich set of pattern formations, motivating further research and developments on intrinsic self-organization on manifolds, and rotation group in particular.


\begin{appendix}
\section{Appendix}
In this appendix, we briefly present basic materials on flows on manifold, interaction velocity field and linear algebra which have been used in the proceeding sections of the paper.

\subsection{Flows on manifolds}
\label{subsect:A-wp}
Consider a smooth, complete and connected $\dim$-dimensional Riemannian manifold $\M$ with the intrinsic distance $d$. Denote the Euclidean distance in $\R^\dim$ by $\| \cdot \|_{\R^\dim}$, and let $T\in(0,\infty]$ denote a generic final time and $\U$ a generic open subset of $\M$. \smallskip

{\em Well-posedness of flow maps.}  Local well-posedness of the flow map equation \eqref{eq:characteristics-general} on an arbitrary manifold $\M$ can be established in local charts using standard ODE theory. The notions of Lipschitz continuity and boundedness on charts of a vector field on $\U$ are defined as follows.
\begin{defn}[Lipschitz continuity and boundedness on charts]\label{defn:Lip}
	Let $X$ be a vector field on $\U$. We say that $X$ is \emph{locally Lipschitz continuous on charts} if for every chart $(U,\varphi)$ of $\M$ and compact set $\Q\subset U \cap \U$, there exists $L_{\varphi,Q}>0$ such that
	\be\label{eq:Lip-vf}
		\norm{\varphi_*X(x) - \varphi_*X(y)}_{\R^\dim} \leq L_{\varphi,\Q} \norm{\varphi(x)-\varphi(y)}_{\R^\dim}, \quad \mbox{for all $x,y\in \Q$},
	\ee
where $\varphi_*$ stands for the push-forward of $\varphi$. We denote by $\norm{X}_{\mathrm{Lip}(\varphi,\Q)}$ the smallest such constant. 

We say that $X$ is \emph{locally bounded on charts} if for every chart $(U,\varphi)$ of $\M$ and compact set $\Q\subset U\cap \U$, there exists $C_{\varphi,\Q}>0$ such that
	\bes
		\norm{\varphi_*X(x)}_{\R^\dim} \leq C_{\varphi,\Q}, \qquad \mbox{for all $x\in \Q$}.
	\ees
We denote by $\norm{X}_{L^\infty(\varphi,\Q)}$ the smallest such constant.
\end{defn}

The local well-posedness of flows generated by locally Lipschitz continuous on charts vector fields is given by the following Cauchy--Lipschitz theorem.
\begin{thm}[Cauchy--Lipschitz]\label{thm:Cauchy-Lip}
	Let $\a\in(0,\infty]$ and let $X$ be a time-dependent vector field on $\U\times [0,\a)$. Suppose that the vector fields in $\{X_t\}_{t\in [0,T)}$ are locally Lipschitz continuous on charts and satisfy, for any chart $(U,\varphi)$ of $\M$ and compact sets $\Q\subset U\cap \U$ and $\calS \subset [0,\a)$,
	\be\label{eq:bdd-Lip}
		\int_\calS \left( \norm{X_t}_{L^\infty(\varphi,\Q)} + \norm{X_t}_{\mathrm{Lip}(\varphi,\Q)} \right) \d t < \infty.
	\ee
Then, for every compact subset $\Sigma$ of $\U$, there exists a unique maximal flow map generated by $(X,\Sigma)$.
\end{thm}
\begin{proof}
The proof can be found in a standard textbook material; see \cite[Chapter~12]{Lee2013} or \cite[Chapter~4]{AMR1988}. A proof is also included with Theorem A.4 in \cite{FePaPa2020}.
\end{proof}

In our context we will apply Theorem \ref{thm:Cauchy-Lip} to $\M = SO(3)$, which is a compact set. The Escape Lemma \cite[Chapter~12]{Lee2013} states that if an integral curve of a Lipschitz continuous vector field on a manifold is not global (i.e., not defined for all $t \in \R$), then the image of that curve cannot lie in any compact subset of the manifold. Consequently, Lipschitz continuous vector fields on compact manifolds that are defined at all times (i.e., $a=\infty$ above) generate global flows. 

For our purposes we will need to restrict the dynamics to certain subsets of $SO(3)$ (e.g., a suitable geodesic disk). For this reason, a global well-posedness result will also need to guarantee that the dynamics remains confined within such a subset. 
Use the notation
\bes
	D_r(p) = \{x\in\M \st d(x,p) < r\}, \quad \mbox{for any $p\in\M$ and $r>0$},
\ees
to denote the open disk in $\M$ of centre $p$ and radius $r$. Also, $\log_x$ below denotes the Riemannian logarithm map at $x \in \M$ (see \cite{Petersen2006}).

The following global version of the Cauchy--Lipschitz theorem will be needed in our study.
\begin{thm}[Global Cauchy--Lipschitz]\label{thm:global-Cauchy-Lip}
	Suppose that $\U$ is geodesically convex and  $\Sigma \subset \U$ is compact. Assume the same hypotheses as in Theorem \ref{thm:Cauchy-Lip}, and in addition,  that $\a=\infty$ and there exist $p\in\U$, and $r_\ast>r>0$ such that $\Sigma\subset \overline{D_r(p)} \subset D_{r_\ast}(p)\subset \U$ and
	\be\label{eq:attractive-general}
		\ap{-\log_x p,X(x,t)}_x \leq 0, \quad \mbox{for all $x\in D_{r_\ast}(p)\setminus D_r(p)$ and $t\in[0,\infty)$}.
	\ee
	Then, there exists a unique flow map $\Psi$ generated by $(X,\Sigma)$ defined on $\Sigma\times [0,\infty)$; furthermore, $\Psi(x,t)\in \overline{D_r(p)}$ for all $(x,t)\in\Sigma\times[0,\infty)$.
\end{thm}
\begin{proof} We refer to \cite[Theorem A.4]{FePaPa2020} for the proof. In informal terms, condition \eqref{eq:attractive-general} states that the vector field $X$ is pointing ``inside'' the disk $D_r(p)$ at points on the boundary $\p D_r(p)$ and some of its outside vicinity. Hence, the dynamics remains contained in $\overline{D_r(p)}$ and also, by the Escape Lemma, the flow is global in time.
\end{proof}


\subsection{Flows for the interaction velocity field}
\label{subsect:A-cont}

We focus exclusively on the velocity field $\V[\rho]$ associated with the interaction equation (see equation \eqref{eqn:v-field}) set up on the rotation group. We fix a curve $\rho \in \Cont([0,T),\P(\U))$ and show that for an interaction potential that satisfies Hypothesis \ref{hyp:K}, $\V[\rho]$ satisfies the assumptions of Theorem \ref{thm:Cauchy-Lip}, and hence it generates a local flow map. This legitimates the definition of the map $\Gamma$ used in Theorem \ref{thm:well-posedness}. We also show that one can apply Theorem \ref{thm:global-Cauchy-Lip} to $\V[\rho]$ when the interaction potential is purely attractive; this will be used in Proposition \ref{prop:inv-cont} to establish the global well-posedness of solutions.

For simplicity, we assume that $\U$ is geodesically convex. In particular this implies that $\U$  can be covered by a single chart, which we will denote by $(\U,\psi)$; such a chart can be given by a normal chart for instance. Note that in our setup $\U = \S$, so this assumption is satisfied. Our first result establishes that the Lipschitz theory given in Theorem \ref{thm:Cauchy-Lip} applies to the interaction velocity field $\V[\rho]$.

\begin{lem}\label{lem:interaction-complete}
	Let $K$ satisfy \ref{hyp:K}, and let $\rho\in \Cont([0,T);\P(\U))$. Then the velocity field $\{\V[\rho](\cdot,t)\}_{t\in [0,T)}$ given by \eqref{eqn:v-field} satisfies the assumptions of the Cauchy--Lipschitz theorem \ref{thm:Cauchy-Lip}.
\end{lem}
\begin{proof}
Let $\Q\subset \U$ be compact.  We first show that the maps $V\mapsto {\|\grad K_V\|}_{L^\infty(\psi,\Q)}$ and $V\mapsto {\|\grad K_V\|}_{\mathrm{Lip}(\psi,\Q)}$ are locally bounded on $\U$. To do this, we will use the fact that the map $(R,Q)\mapsto \grad d^2_Q(R)$ is smooth on $\U\times\U$. Indeed, for all $R\in \Q$ and $Q\in \U$ we get
	\bes
		\norm{\psi_*\grad K_Q(R)}_{\R^\dim} \leq |g'(d(R,Q)^2)| \|\psi_* \grad d^2_Q(R)\|_{\R^\dim} \leq |g'(d(R,Q)^2)| \|\grad d^2_Q\|_{L^\infty(\psi,\Q)},
	\ees
and by the local boundedness of $g'$ and of $Q \mapsto {\|\grad d^2_Q\|}_{L^\infty(\psi,\Q)}$,  we get
\[ Q\mapsto \norm{\grad K_Q}_{L^\infty(\psi,\Q)} \quad \mbox{is locally bounded}. \] 
Also, for all $R,Q \in \Q$ and $V\in\U$ we have
	\begin{align*}
	\begin{aligned}
		& \norm{\psi_*\grad K_V(R) - \psi_*\grad K_V(Q)}_{\R^\dim} \\
		& \hspace{1cm} = \norm{g'(d(R,V)^2) \psi_*\grad d^2_V(R)  -  g'(d(Q,V)^2) \psi_*\grad d^2_V(Q)}_{\R^\dim}\\
		& \hspace{1cm} \leq |g'(d(R,V)^2)| {\|\psi_*\grad d^2_V(R) - \psi_*\grad d^2_V(Q)\|}_{\R^\dim} \\
		& \hspace{1.2cm}+ {\|\psi_*\grad d^2_V(Q)\|}_{\R^\dim} |g'(d(R,V)^2) - g'(d(Q,V)^2)| \\
		& \hspace{1cm} \leq  |g'(d(R,V)^2)| {\|\grad d^2_V\|}_{\mathrm{Lip}(\psi,\Q)}  {\|\psi(R)-\psi(Q)\|}_{\R^\dim}  \\
		& \hspace{1.2cm}+ {\|\grad d^2_V\|}_{L^\infty(\psi,\Q)} |g'(d(R,V)^2) - g'(d(Q,V)^2)|.
	\end{aligned}
	\end{align*}
Using the local Lipschitz continuity and the local boundedness of $g'$ (in fact of $r\mapsto g'(r^2)$), and the local boundedness of the maps $V \mapsto {\|\grad d^2_V\|}_{\mathrm{Lip}(\psi,\Q)}$ and $V \mapsto {\|\grad d^2_V\|}_{L^\infty(\psi,\Q)}$, we conclude 
\[ V \mapsto {\|\grad K_V\|}_{\mathrm{Lip}(\psi,\Q)} \quad \mbox{is locally bounded}. \]

Now let $\Q_t \subset \U$ be a compact set containing $\supp(\rho_t)$ such that $t\mapsto \mathrm{diam}(\Q_t)$ is nondecreasing.  Then, for all $R,Q\in \Q$ and $t\in \calS$, where $\calS\subset [0,T)$ is compact, we have
\be\label{eq:bdd-int-vel}
	\norm{\psi_* \V[\rho](R,t)}_{\R^\dim} \leq \int_{\Q_t} \norm{\psi_*\grad K_Q(R)}_{\R^\dim} \d \rho_t(Q) \leq \sup_{\bar Q\in \Q_s} \norm{\grad K_{\bar Q}}_{L^\infty(\psi,\Q)},
\ee
where $s=\sup(\calS)$, and
\begin{align} 
\begin{aligned}\label{eq:Lip-int-vel}
& \norm{\psi_*\V[\rho](R,t) - \psi_*\V[\rho](Q,t)}_{\R^\dim} \\
& \hspace{1cm} \leq \int_{\Q_t} \norm{\psi_*\grad K_V(R) - \psi_*\grad K_V(Q))}_{\R^\dim} \d \rho_t(V)\\
& \hspace{1cm} \leq \sup_{\bar V \in \Q_s} \norm{\grad K_{\bar V}}_{\text{Lip}(\psi,\Q)} {\|\psi(R)-\psi(Q)\|}_{\R^\dim}.
\end{aligned}
\end{align}
Now, the proof follows from \eqref{eq:bdd-int-vel} and \eqref{eq:Lip-int-vel}.
\end{proof}

\begin{rem}\label{rem:indep-max-time} We make a key observation that since the $L^\infty$ and Lipschitz bounds in \eqref{eq:bdd-int-vel} and \eqref{eq:Lip-int-vel} do not depend on $\rho$, the maximal time of existence of the flow map generated by $v[\rho]$ does not depend on the curve $\rho$. 
\end{rem}

The global version of the Cauchy-Lipschitz theory, Theorem \ref{thm:global-Cauchy-Lip}, also applies to the interaction velocity field when the potential $K$ is purely attractive ($g'\geq 0$).  The results is given by the following lemma.

\begin{lem}\label{lem:interaction-complete-global}
	 Let $K$ satisfy \ref{hyp:K} with $g'\geq 0$, and let $\rho\in \Cont([0,\infty);\P(\U))$ fixed. Let $\Sigma\subset \U$ be compact and such that $\Sigma\subset \overline{D_r(R)} \subset D_{r_\ast}(R) \subset \U$ for some $R\in\U$ and $0<r<r_\ast<\pi/2$. Then, the pair $(\V[\rho],\Sigma)$ satisfies the assumptions of the global Cauchy--Lipschitz theorem \ref{thm:global-Cauchy-Lip} provided $\supp(\rho_t)\subset \overline{D_r(R)}$ for all $t\in[0,\infty)$.
\end{lem}
\begin{proof} We need to check that $\V[\rho]$ verifies \eqref{eq:attractive-general}. Suppose that $\supp(\rho_t)\subset \overline{D_r(R)}$ for all $t\in[0,\infty)$ and let $Q\in D_{r_\ast}(R)\setminus D_r(R)$. Then, for all $t\in[0,\infty)$ one has
\begin{align}
\begin{aligned} \label{eq:integral-global-Lipschitz-velocity}
& -\log_Q R \cdot \V[\rho](Q,t)  \\
& \hspace{1cm} = \log_Q R \cdot \grad K*\rho_t(Q) = \int_{\overline {D_r(R)}} g'(d(Q,V)^2) \log_Q R \cdot \grad d_V^2(Q) \d \rho_t(V)  \\
&\hspace{1cm}  = -2\int_{\overline {D_r(R)}} g'(d(Q,V)^2) \log_Q R \cdot \log_Q V \d \rho_t(V).
\end{aligned}
\end{align}
Let $V \in \overline{D_r(R)} \subset D_{d(Q,R)}(R)$ be fixed. Then, since $d(Q,R)<r_\ast<\pi/2$, the closed disk $\overline{D_{d(Q,R)}(R)}$ is geodesically convex. Let $\gamma:[0,1]\to \overline{D_{d(Q,R)}(R)}$ be the unique minimizing geodesic connecting $Q$ to $V$. Then, 
\bes
	\left. \frac{\der}{\der t}\right|_{t = 0} d(\gamma(t),R)^2 = \grad d_R^2(\gamma(0)) \cdot \gamma'(0) = -2 \log_Q R \cdot \log_Q V.
\ees
Note that $\left. \frac{\der}{\der t}\right|_{t = 0} d(\gamma(t),R)^2 \leq 0$. Indeed, otherwise there would exist $\tau\in(0,1)$ such that $d(\gamma(\tau),R)>d(\gamma(0),R)=d(Q,R)$, which contradicts $\gamma([0,1])\subset \overline{D_{d(Q,R)}(R)}$ and thus the geodesic convexity of $\overline{D_{d(Q,R)}(R)}$.

We use these considerations in \eqref{eq:integral-global-Lipschitz-velocity}, we get that for attractive potentials ($g' \geq 0$),
\[
-\log_Q R \cdot \V[\rho](Q,t) \leq 0, \quad \text{ for all } Q\in D_{r_\ast}(R)\setminus D_r(R) \text{ and } t\in[0,\infty),
\] 
which is the required condition in Theorem \ref{thm:global-Cauchy-Lip}.
\end{proof}

\subsection{Some linear algebra results}
\label{subsect:A-matrices}
The following technical lemmas are used in the proofs of some of the main results.

\begin{lem}\label{L1}
For any two matrices $A,B \in \R^{\dim \times \dim}$, one has
\[
{\|AB\|}_F\leq {\|A\|}_F {\|B\|}_F.
\]
\end{lem}
\begin{proof}
We use the definition of Frobenius nrom and the Cauchy-Schwarz inequality to find
\begin{align*}
{\|AB\|}_F^2&=\sum_{i, j=1}^n[AB]_{ij}^2=\sum_{i, j=1}^n\left(\sum_{k=1}^n[A]_{ik}[B]_{kj}\right)^2\\
&\leq\sum_{i,j=1}^n\left(\sum_{k=1}^n[A]_{ik}^2\right)\left(\sum_{k=1}^n[B]_{kj}^2\right)={\|A\|}_F^2 {\|B\|}_F^2.
\end{align*}
\end{proof}

\begin{lem}
\label{lem:l-bound}
Let $R \in SO(3)$ be a rotation matrix with $0\leq d(I,R)<\pi/2$ and $V \in \R^{3 \times 3}$. Then the following inequality holds:
\[
\mathrm{tr}(RVV^T)\geq \cos d(I,R)\cdot\|V\|_{F}^2.
\]
\end{lem}
\begin{proof}
Denote $d(I,R)=\theta$. One can find a basis and a matrix $Q$ such that
\[
R=Q\begin{bmatrix}
1&0&0\\
0&\cos\theta&-\sin\theta\\
0&\sin\theta&\cos\theta
\end{bmatrix}Q^T.
\]
Then we have
\begin{align}
\label{eqn:tr-3m}
\mathrm{tr}(RVV^T)&=\mathrm{tr}\left(
\begin{bmatrix}
1&0&0\\
0&\cos\theta&-\sin\theta\\
0&\sin\theta&\cos\theta
\end{bmatrix}Q^TVV^TQ
\right)   \\
&=\mathrm{tr}\left(
\begin{bmatrix}
1&0&0\\
0&\cos\theta&-\sin\theta\\
0&\sin\theta&\cos\theta
\end{bmatrix}(Q^TV)(Q^TV)^T
\right).
\end{align}
On the other hand, the first term inside the trace can be rewritten as 
\begin{equation}
\label{eqn:decomp}
\begin{bmatrix}
1&0&0\\
0&\cos\theta&-\sin\theta\\
0&\sin\theta&\cos\theta
\end{bmatrix} = \cos \theta
\begin{bmatrix}
1&0&0\\
0&1&0\\
0&0&1
\end{bmatrix} + (1-\cos \theta)
\begin{bmatrix}
1&0&0\\
0&0&0\\
0&0&0
\end{bmatrix} +
\begin{bmatrix}
0&0&0\\
0&0&-\sin\theta\\
0&\sin\theta&0
\end{bmatrix}.
\end{equation}
First note that the third matrix on the right-hand-side above is skew-symmetric and its contribution in \eqref{eqn:tr-3m} is zero. Indeed, for any skew-symmetric matrix $A$ and any matrix $M$, $\mathrm{tr}(AMM^T) =0$, as
\begin{align*}
\mathrm{tr}(AMM^T)=\mathrm{tr}((AMM^T)^T)=\mathrm{tr}(MM^TA^T)=\mathrm{tr}(A^TMM^T)=-\mathrm{tr}(AMM^T).
\end{align*}

The contribution of the second term in the right-hand-side of \eqref{eqn:decomp} is non-negative. Indeed, let $[Q^TV]_{ij}=b_{ij}$, and compute:
\[
\mathrm{tr}\left(\begin{bmatrix}
1&0&0\\
0&0&0\\
0&0&0
\end{bmatrix}(Q^TV)(Q^TV)^T\right)=\sum_{i, j, k}\left(\begin{bmatrix}
1&0&0\\
0&0&0\\
0&0&0
\end{bmatrix}_{ij}[Q^TV]_{jk}[(Q^TV)^T]_{ki}\right)=\sum_{k}b_{1k}^2\geq0.
\]
Hence, it follows from \eqref{eqn:decomp} that
\[
\mathrm{tr}(RVV^T) \geq  \cos \theta \, \mathrm{tr} \left( (Q^TV)(Q^TV)^T\right) \geq\cos\theta\|Q^TV\|_F^2=\cos\theta\|V\|_F^2.
\]
This proves the claim.
\end{proof}
\end{appendix}


{\thanks{\textbf{Acknowledgments.}  R.F. was supported by NSERC Discovery Grant PIN-341834 during this research. The work of S.-Y. Ha was supported by National Research Foundation of Korea (NRF-2020R1A2C3A01003881), and the work of H. Park was supported by Basic Science Research Program through the National Research Foundation of Korea, funded by the Ministry of Education (2019R1I1A1A01059585). The initial part of the work in this paper was done while H. Park visited Simon Fraser University, a visit supported through a combination of the above mentioned grants.}


\def\cprime{$'$}

\end{document}